\documentclass[a4paper]{amsart}

\usepackage{amssymb,amscd}
\usepackage{amsmath,amsfonts,latexsym}
\usepackage[dvips]{graphicx}
\usepackage{color}
\usepackage[english]{babel}
\usepackage[latin1]{inputenc}
\usepackage{enumerate,afterpage}
\usepackage{tikz}
\usetikzlibrary{matrix,arrows}
\usepackage[all]{xy}

\usepackage{mathrsfs}
\usepackage{mathtools}
\usepackage{stmaryrd}

\newcounter{notes}%

\newtheorem{cor}{Corollary}[section]
\newtheorem{theorem}[cor]{Theorem}
\newtheorem{prop}[cor]{Proposition}
\newtheorem{lemma}[cor]{Lemma}

\newtheorem{claim}[cor]{Claim}

 \newtheorem{Lemma}[cor]{Lemma}
 \newtheorem{Proposition}[cor]{Proposition}

\theoremstyle{definition}
\newtheorem{defi}[cor]{Definition}

\newtheorem{Remark}[cor]{Remark}
\newtheorem{remark}[cor]{Remark}

\newcommand{\HS}{\mathrm{HS}}

\newcommand{\HH}{{\mathbb H}}

\newcommand{\N}{{\mathbb N}}
\newcommand{\HP}{{\mathbb{HP}}}
\newcommand{\R}{{\mathbb R}}

\newcommand{\AdS}{\mathbb{A}\mathrm{d}\mathbb{S}}
\newcommand{\dS}{\mathrm{d}\mathbb{S}}

\newcommand{\SL}{\mathrm{SL}}
\newcommand{\PSL}{\mathrm{PSL}}

\newcommand{\Teich}{\mathscr{T}}
\newcommand{\Angles}{\mathcal{A}}
\newcommand{\XX}{\mathbb X}
\newcommand{\BB}{\mathcal B}
\newcommand{\RR}{\mathbb R}
\newcommand{\CC}{\mathbb C}
\newcommand{\CP}{\mathbb{CP}}
\newcommand{\Rtau}{\mathbb R + \mathbb R \tau}
\newcommand{\Rsigma}{\mathbb R + \mathbb R \sigma}
\newcommand{\RP}{\mathbb{RP}}

\newcommand{\OO}{\operatorname{O}}
\newcommand{\PSO}{\operatorname{PSO}}
\newcommand{\PO}{\operatorname{PO}}
\newcommand{\GL}{\operatorname{GL}}
\newcommand{\PGL}{\operatorname{PGL}}
\newcommand{\Herm}{\operatorname{Herm}}
\newcommand{\minimatrix}[4]{\left(\begin{matrix} #1 & #2 \\ #3 & #4 \end{matrix}\right)  }
\newcommand{\bminimatrix}[4]{\begin{bmatrix} #1 & #2 \\ #3 & #4 \end{bmatrix}  }

\newcommand{\btwovector}[2]{\begin{bmatrix} #1\\#2 \end{bmatrix} }

\newcommand{\rot}{\operatorname{rot}}
\newcommand{\wt}{\widetilde}
\newcommand{\plane}{\mathscr{P}}
\newcommand{\doubles}{\mathscr{D}}
\newcommand{\poly}{\mathsf{polyg}}
\newcommand{\ML}{\mathcal{ML}}
\newcommand{\Graph}{\mathsf{Graph}}

\newcommand{\AdSPoly}{\mathsf{AdSPolyh}}
\newcommand{\PsiAdS}{\Psi^{\mathsf{AdS}}}

\newcommand{\HPPoly}{\mathsf{HPPolyh}}
\newcommand{\PsiHP}{\Psi^{\mathsf{HP}}}

\let\oldtocsection=\tocsection

\let\oldtocsubsection=\tocsubsection

\let\oldtocsubsubsection=\tocsubsubsection

\renewcommand{\tocsection}[2]{\hspace{0em}\oldtocsection{#1}{#2}}
\renewcommand{\tocsubsection}[2]{\hspace{1em}\oldtocsubsection{#1}{#2}}
\renewcommand{\tocsubsubsection}[2]{\hspace{2em}\oldtocsubsubsection{#1}{#2}}

\begin{document}

\title{Polyhedra inscribed in a quadric}

\author{Jeffrey Danciger}
\address{Department of Mathematics, University of Texas at Austin}
\email{jdanciger@math.utexas.edu}
\urladdr{www.ma.utexas.edu/users/jdanciger}

\author{Sara Maloni}
\address{Department of Mathematics, Brown University}
\email{sara\_maloni@brown.edu}
\urladdr{http://www.math.brown.edu/$\sim$maloni}

\author{Jean-Marc Schlenker}
\address{Department of mathematics, University of Luxembourg}
\email{jean-marc.schlenker@uni.lu}
\urladdr{math.uni.lu/schlenker}

\thanks{The first author was partially supported by the National Science Foundation under the grant DMS 1103939. The second author was partially supported by the European Research Council under the {\em European Community}'s seventh Framework Programme (FP7/2007-2013)/ERC {\em grant agreement}}

\date{\today}

\begin{abstract}
We study convex polyhedra in three-space that are inscribed in a quadric surface. Up to projective transformations, there are three such surfaces: the sphere, the hyperboloid, and the cylinder. Our main result is that a planar graph $\Gamma$ is realized as the $1$--skeleton of a polyhedron inscribed in the hyperboloid or cylinder if and only if $\Gamma$ is realized as the $1$--skeleton of a polyhedron inscribed in the sphere and $\Gamma$ admits a Hamiltonian cycle.

Rivin characterized convex polyhedra inscribed in the sphere by studying the geometry of ideal polyhedra in hyperbolic space. We study the case of the hyperboloid and the cylinder by parameterizing the space of convex ideal polyhedra in anti-de Sitter geometry and in half-pipe geometry. Just as the cylinder can be seen as a degeneration of the sphere and the hyperboloid, half-pipe geometry is naturally a limit of both hyperbolic and anti-de Sitter geometry. We promote a unified point of view to the study of the three cases throughout.
\end{abstract}

\maketitle

\tableofcontents

\section{Introduction and results}

\subsection{Polyhedra inscribed in a quadric}

According to a celebrated result of Steinitz (see e.g. \cite[Chapter 4]{ziegler:lectures}), 
a graph $\Gamma$ is the $1$--skeleton of a convex polyhedron in $\R^3$ if and
only if $\Gamma$ is planar and $3$--connected. Steinitz \cite{ste_iso} also discovered, however, that there exists a $3$--connected planar graph
which is not realized as the $1$--skeleton of any polyhedron inscribed in the unit sphere~$S$,
answering a question asked by Steiner
\cite{ste_sys} in 1832. An understanding of which polyhedral types can or can not be inscribed in the sphere remained elusive until Hodgson, Rivin, and Smith~\cite{hodgson1992characterization} gave a full characterization in 1992. This article is concerned with realizability by polyhedra inscribed 
in other quadric surfaces in $\R^3$. Up to projective transformations, there are two such surfaces: 
the hyperboloid $H$, defined by $x_1^2 + x_2^2 - x_3^2 = 1$, and the cylinder $C$, defined 
by $x_1^2 + x_2^2 = 1$ (with $x_3$ free).

\begin{defi}\label{def:inscribed}
A convex polyhedron $P$ is {\it inscribed} in the hyperboloid $H$ (resp. the cylinder $C$) if $P\cap H$ (resp. $P \cap C$) is exactly the set
of vertices of $P$. 
\end{defi}

\noindent If a polyhedron $P$ is inscribed in the cylinder $C$, then $P$ lies in the solid cylinder $x_1^2 + x_2^2 \leq 1$ (and $x_3$ free), with all points of $P$ except its vertices lying in the interior.
A polyhedron $P$ inscribed in the hyperboloid $H$ could lie in (the closure of) either complementary region of $\RR^3 \setminus H$. However, after performing a projective transformation, preserving $H$ and exchanging the two complementary regions of $\RR^3 \setminus H$, we may (and will henceforth) assume that all points of $P$, except its vertices, lie in the interior of the solid hyperboloid $x_1^2 +x_2^2 - x_3^2 \leq 1$.

Recall that a {\it Hamiltonian cycle} in  is a closed path visiting each
vertex exactly once. We prove the following.

\begin{theorem}
\label{thm:main}
Let $\Gamma$ be a planar graph. Then the following conditions are equivalent:
\begin{enumerate}
\item[(C):] $\Gamma$ is the $1$--skeleton of some convex polyhedron inscribed in the cylinder.
\item[(H):] $\Gamma$ is the $1$--skeleton of some convex polyhedron inscribed in the hyperboloid.
\item[(S):]  $\Gamma$ is the $1$--skeleton of some convex  polyhedron inscribed in the sphere and $\Gamma$ admits a Hamiltonian cycle.
\end{enumerate}
\end{theorem}

The ball $x_1^2 + x_2^2 + x_3^2 < 1$, thought of as lying in an affine chart $\RR^3$ of $\RP^3$, gives the projective model for hyperbolic space $\HH^3$, with the sphere $S$ describing the ideal boundary $\partial_\infty \HH^3$. 
In this model, projective lines and planes intersecting the ball correspond to totally geodesic lines and planes in $\HH^3$. Therefore a convex polyhedron inscribed in the sphere is naturally associated to a \emph{convex ideal polyhedron} in the hyperbolic space $\HH^3$. 

Following the pioneering work of Andreev~\cite{Andreev, Andreev-ideal}, Rivin~\cite{riv_ach} gave a parameterization of the deformation space of such ideal polyhedra in terms of dihedral angles. As a corollary, Hodgson, Rivin and Smith \cite{hodgson1992characterization} showed that deciding whether a planar graph $\Gamma$ may be realized as the $1$--skeleton of a polyhedron inscribed in the sphere amounts to solving a linear programming problem on~$\Gamma$. To prove Theorem~\ref{thm:main}, we show that, given a Hamiltonian path in $\Gamma$, there is a similar linear programming problem whose solutions determine polyhedra inscribed in either the cylinder or the hyperboloid. 

The solid hyperboloid $x_1^2 + x_2^2 - x_3^2 < 1$ in $\RR^3$ gives a picture of the projective model for \emph{anti-de Sitter} (AdS) geometry in an affine chart.
Therefore a convex polyhedron inscribed in the hyperboloid is naturally associated to a convex ideal polyhedron in the anti-de Sitter space $\AdS^3$, which is a Lorentzian analogue of hyperbolic space. Similarly, the solid cylinder $x_1^2 + x_2^2 < 1$ (with $x_3$ free) in an affine chart $\RR^3$ of $\RP^3$ gives the projective model for \emph{half-pipe} (HP) geometry. Therefore a convex polyhedron inscribed in the cylinder is naturally associated to a convex ideal polyhedron in the half-pipe space $\HP^3$. Half-pipe geometry, introduced by Danciger~\cite{dan_geo, dan_age, dan_ide}, is a transitional geometry which, in a natural sense, is a limit of both hyperbolic and anti-de Sitter geometry.
 In order to prove Theorem~\ref{thm:main} we study the deformation spaces of ideal polyhedra in both $\AdS^3$ and $\HP^3$ concurrently. By viewing polyhedra in $\HP^3$ as limits of polyhedra in both $\HH^3$ and $\AdS^3$, we are able to translate some geometric information between the three settings. In fact we are able to give parameterizations (Theorems~\ref{thm:main-AdS-angles}, \ref{thm:main-AdS-metrics} and Theorem~\ref{thm:main-HP}) of the spaces of ideal polyhedra in both $\AdS^3$ and $\HP^3$ in terms of geometric features of the polyhedra.
 This, in turn, describes the moduli of convex polyhedra inscribed in the hyperboloid and the moduli of convex polyhedra inscribed in the cylinder, where polyhedra are considered up to projective transformations fixing the respective quadric. It is these parameterizations which should be considered the main results of this article; Theorem~\ref{thm:main} will follow as a corollary.
 
 \begin{figure}[h]
 \includegraphics[width = 4.0in]{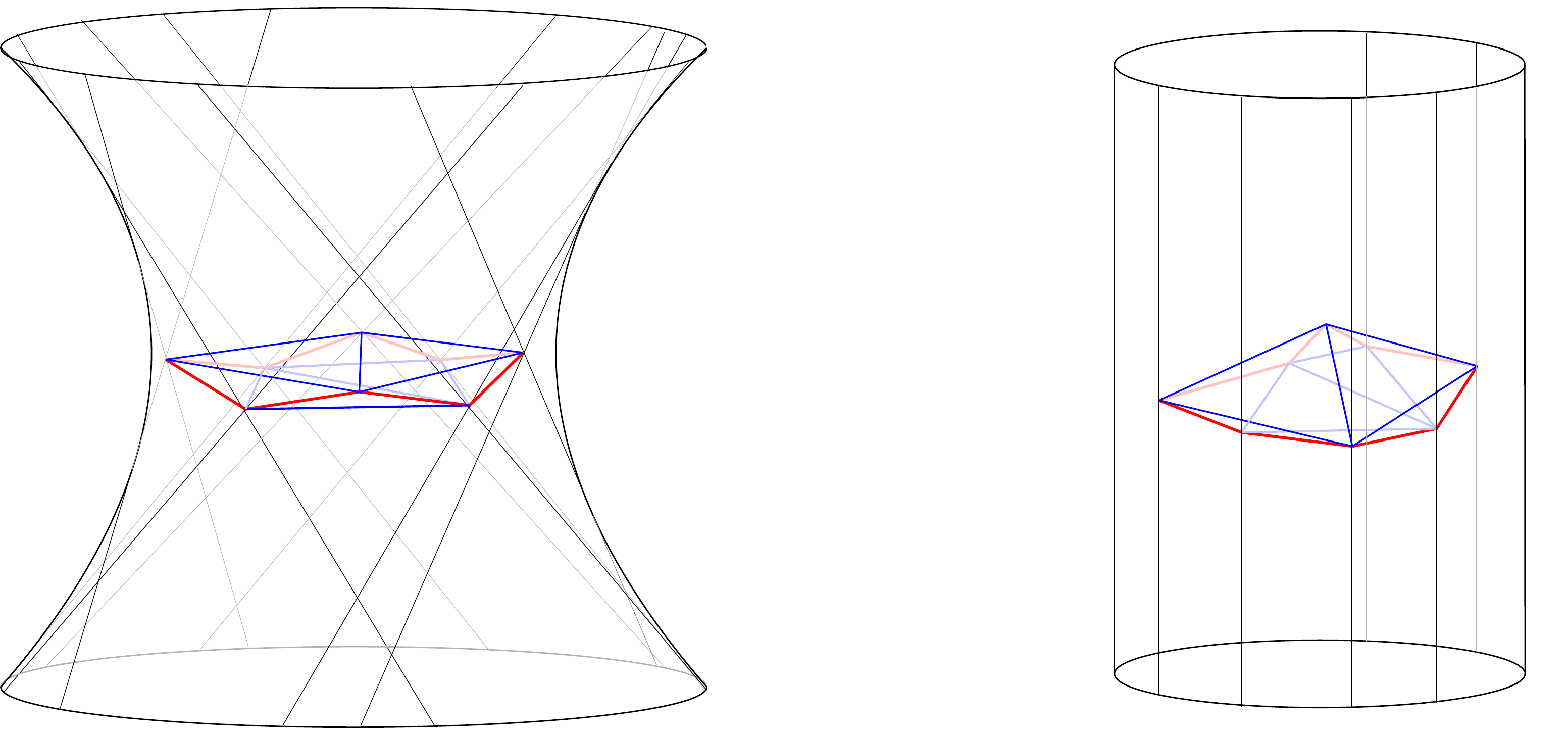}
 \caption{\label{fig:hyperboloid-and-cylinder} A polyhedron inscribed in the hyperboloid (left) and a combinatorial equivalent polyhedron inscribed in the cylinder (right). The $1$--skeleton of any such polyhedron admits a Hamiltonian which we call the \emph{equator} (red). }
 \end{figure}
 
\subsection{Rivin's two parameterizations of ideal polyhedra in $\HH^3$}\label{sec:rivin_param}

Rivin gave two natural parameterizations of the space of convex ideal polyhedra in the hyperbolic space~$\HH^3$. Let $P$ be a convex ideal polyhedron in $\HH^3$, let $P^*$ denote the Poincar\'e dual of $P$, and let $E$ denote the set of edges of the $1$--skeleton of $P^*$ (or of $P$). Then the function $\theta \in \RR^E$ assigning to each edge $e^*$ of $P^*$ the dihedral angle at the corresponding edge $e$ of $P$ satisfies the following three conditions:
 
 \begin{enumerate}
\item  $0 < \theta(e^*) < \pi$ for all edges $e^*$ of $P^*$.
\item If $e_1^*, \ldots, e_k^*$ bound a face of $P^*$, then $\theta(e_1^*) + \cdots + \theta(e_k^*) = 2 \pi$.
\item  If $e_1^*, \ldots, e_k^*$ form a simple circuit which does not bound a face of $P^*$, then $\theta(e_1^*) + \cdots + \theta(e_k^*) > 2 \pi$.
\end{enumerate}
Rivin \cite{riv_ach} shows that, for an abstract polyhedron $P$, any assignment of weights $\theta$ to the edges of $P^*$ that satisfy the above three conditions is realized as the dihedral angles of a unique (up to isometries) non-degenerate ideal polyhedron in~$\HH^3$. Further the map taking any ideal polyhedron $P$ to its dihedral angles $\theta$ is a homeomorphism onto the complex of all weighted planar graphs satisfying the above linear conditions. This was first shown by Andreev~\cite{Andreev-ideal} in the case that all angles are acute.

The second parameterization~\cite{MR1280952} characterizes an ideal polyhedron $P$ in terms of the geometry intrinsic to the surface of the boundary of $P$. The path metric on $\partial P$, called the \emph{induced metric}, is a complete hyperbolic metric on the $N$-times punctured sphere $\Sigma_{0,N}$, which determines a point in the Teichm\"uller space $\Teich_{0,N}$. Rivin also shows that the map taking an ideal polyehdron to its induced metric is a homeomorphism onto $\Teich_{0,N}$.
 
\subsection{Two parameterizations of ideal polyhedra in $\AdS^3$} \label{sec:intro-AdS-param}

Anti-de Sitter geometry is a Lorentzian analogue of hyperbolic geometry in the sense that the anti-de Sitter space $\AdS^n$ has all sectional curvatures equal to~$-1$. However, the metric is Lorentzian (meaning indefinite of signature $(n-1, 1)$), making the geometry harder to work with than hyperbolic geometry, in many cases. For our purposes, it is most natural to work with the projective model of $\AdS^3$ (see Section~\ref{ads_background}), which identifies $\AdS^3$ with an open region in $\RP^3$, and its ideal boundary $\partial_\infty \AdS^3$ with the boundary of that region. The intersection of $\AdS^3$ with an affine chart is the region $x_1^2 + x_2^2 - x_3^2 < 1$ bounded by the hyperboloid $H$. The ideal boundary $\partial_\infty \AdS^3$, seen in this affine chart, is exactly $H$.

Let $P$ be a convex ideal polyhedron in $\AdS^3$ with $N$ vertices. That $P$ is ideal means that the closure of $P$ in $\AdS^3 \cup \partial_\infty \AdS^3$ is a polyhedron whose intersection with $\partial_\infty \AdS^3$ is precisely its vertices. That $P$ is convex means that after removing a space-like plane in its complement, $P$ is geodesically convex. Alternatively, $P$ is convex if and only if it is convex in some affine chart of $\RP^3$. Unlike in the hyperbolic setting, there are restrictions (Proposition~\ref{prop:cyclic-order}) on the positions of the $N$ vertices. Some choices of $N$ vertices on the ideal boundary $\partial_\infty \AdS^3$ do not determine an ideal polyhedron. Roughly, this is because the hyperboloid $H$ has mixed curvature and the convex hull of a collection of vertices on $H$ may contain points both inside and outside of $H$. All facets of $P$ are \emph{spacelike}, meaning the restriction of the AdS metric is positive definite. Therefore, by equipping $\AdS^3$ with a time-orientation, we may sort the faces of $P$ into two types, those whose normal is future-directed, and those whose normal is past-directed. The future-directed faces unite to form a disk (a bent ideal polygon), as do the past-directed faces. The edges which separate the past faces from the future faces form a Hamiltonian cycle, which we will refer to as the \emph{equator} of $P$. A \emph{marking} of $P$ will refer to an identification, up to isotopy, of the equator of $P$ with the standard $N$-cycle graph so that the induced ordering of the vertices is positive with respect to the orientation and time orientation of $\AdS^3$.
We let $\AdSPoly = \AdSPoly_N$ denote the space of all marked, non-degenerate convex ideal polyhedra in $\AdS^3$ with $N$ vertices, considered up to orientation and time-orientation preserving isometries of $\AdS^3$. The term \emph{ideal polyhedron in }$\AdS^3$ will henceforth refer to an element of this space. Let $\Sigma_{0,N}$ denote the $N$-punctured sphere. Fix an orientation on $\Sigma_{0,N}$, a simple loop $\gamma$ visiting each puncture once and label the punctures in order along the path. We call the polygon on the positive side of $\gamma$ the \emph{top} and the polygon on the negative side the \emph{bottom} of $\Sigma_{0,N}$. Then, each ideal polyhedron $P$ is naturally identified with $\Sigma_{0,N}$ via the (isotopy class of the) map taking each ideal vertex to the corresponding puncture and the equator to~$\gamma$. This identifies the union of the future faces of $P$ with the top of $\Sigma_{0,N}$ and the past faces with the bottom. See Figure~\ref{fig:marking}. We let $\Graph(\Sigma_{0,N}, \gamma)$ denote the collection of three-connected graphs embedded in $\Sigma_{0,N}$, up to isotopy, each of whose edges connects two distinct punctures and whose edge set contains the edges of $\gamma$. Via the marking, any ideal polyhedron $P$ \emph{realizes} the edges of a graph $\Gamma \in\Graph(N, \gamma)$ as a collection of geodesic lines either on the surface of or inside of $P$.

\begin{figure}[h]
{
\centering

\def\svgwidth{10.0cm}
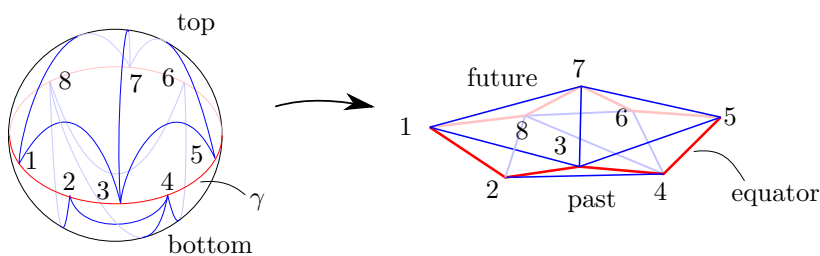

}

\caption{A marking of an ideal polyhedron $P$ (right) in $\AdS^3$ is a labeling of the ideal vertices in order going around the equator in the positive direction. It defines an identification of $\Sigma_{0,N}$ with $P$ that takes $\gamma$ to the equator (red) and the top (resp. bottom) hemisphere of $\Sigma_{0,N}$ (left) to the union of the future (resp. past) faces of~$P$. The $1$--skeleton of $P$ (right, blue and red) defines a graph $\Gamma \in \Graph(\Sigma_{0,N}, \gamma)$ (left, blue and red). \label{fig:marking}}

\end{figure}

Consider a space-like oriented piece-wise totally geodesic surface in $\AdS^3$ and let $T$ and $T'$ be two faces of this surface meeting along a common edge $e$. We measure the \emph{exterior dihedral angle} at $e$ as follows.  The group of isometries of $\AdS^3$ that point-wise fix the space-like line $e$ is a copy of $\OO(1,1)$, which should be thought of as the group of \emph{hyperbolic rotations} or \emph{Lorentz boosts} of the time-like plane orthogonal to~$e$. By contrast to the setting of hyperbolic (Riemannian) geometry, $\OO(1,1)$ has two non-compact components. Therefore there are two distinct types of dihedral angles possible, each of which is described by a real number rather than an element of the circle. Let $\varphi$ be the amount of hyperbolic rotation needed to rotate the plane of $T'$ into the plane of $T$. The sign of $\varphi$ is defined as follows. The light-cone of~$e$ locally divides $\AdS^3$ into four quadrants, two of which are space-like and two of which are time-like. If $T$ and $T'$ lie in opposite space-like quadrants, then we take $\varphi$ to be non-negative, if the surface is convex along $e$, and negative, if the surface is concave along $e$. If $T$ and $T'$ lie in the same space-like quadrant, we take $\varphi$ to be non-positive, if the surface is convex at $e$, and positive, if the surface is concave at~$e$. Therefore, the dihedral angles along the equator of a convex ideal polyhedron $P$ are negative, while the dihedral angles along the other edges are positive. Note that this definition of angle, and in particular the sign convention, agrees with a natural alternative definition in terms of cross-ratios (see Section~\ref{sec:models}). Let $P^*$ and $E$ be as before. We will show (Proposition~\ref{prop:imageinA-AdS}) that the function $\theta \in \RR^E$ assigning to each edge $e^*$ of $P^*$ the dihedral angle at the corresponding edge $e$ of $P$ satisfies the following three conditions:

\begin{enumerate}[(i)]
\item  $ \theta(e^*) < 0$ if $e$ is an edge of the equator $\gamma$, and $\theta(e^*) > 0$ otherwise.
\item If $e_1^*, \ldots, e_k^*$ bound a face of $P^*$, then $\theta(e_1^*) + \cdots + \theta(e_k^*) = 0$.
\item  If $e_1^*, \ldots, e_k^*$ form a simple circuit which does not bound a face of $P^*$, and such that exactly two of the edges are dual to edges of the equator, then $\theta(e_1^*) + \cdots + \theta(e_k^*) > 0$.
\end{enumerate}

Let $\Gamma \in\Graph(N, \gamma)$. Then, thinking of $\Gamma$ as the $1$--skeleton of an abstract polyhedron $P$, we define $\Angles_{\Gamma}$ to be the space of all functions $\theta \in \RR^E$ which satisfy the above three conditions.  Define $\AdSPoly_{\Gamma}$ to be the space of ideal polyhedra in $\AdS^3$ with $1$--skeleton identified with $\Gamma$, and let $\PsiAdS_\Gamma: \AdSPoly_\Gamma \to \Angles_\Gamma$ denote the map assigning to an ideal polyhedron its dihedral angles. All of the maps $\PsiAdS_{\Gamma}$ may be stitched together into one. Let $\Angles$ denote the disjoint union of all $\Angles_{\Gamma}$ glued together along faces corresponding to common subgraphs. Then, we show:

\begin{theorem}
\label{thm:main-AdS-angles}
The map $\PsiAdS: \AdSPoly \to \Angles$, defined by $\PsiAdS(P) = \PsiAdS_{\Gamma}(P)$ if $P \in \AdSPoly_{\Gamma}$, is a homeomorphism.
\end{theorem}

The equivalence of conditions (H) and (S) in Theorem~\ref{thm:main} follows directly from this theorem and from Rivin's theorem (see Section \ref{sec:rivin_param}). Indeed, it is an easy exercise in basic arithmetic to convert any weight function $\theta \in \Angles_{\Gamma}$ into one that satisfies conditions (1), (2), and (3) of Rivin's theorem. To convert any weight function on the edges of a graph $\Gamma$ that satisfies Rivin's conditions into a weight function satisfying our conditions (i), (ii), and (iii) (which define $\Angles_{\Gamma}$) is also easy, provided there is a Hamiltonian cycle $\gamma$ in the $1$--skeleton. See Section \ref{sec:proof_main} for the detailed proof.

We also give a second parameterization of ideal polyhedra in terms of the geometry intrinsic to their boundaries. 
Here we parameterize the space $\overline{\AdSPoly}_N = \AdSPoly_N \cup \poly_N$ of all marked polyhedra with $N$ vertices including both the non-degenerate polyhedra $\AdSPoly_N$ and the degenerate (or collapsed) polyhedra, parameterized by the space $\poly_N$ of marked ideal polygons in $\HH^2$ with $N$ vertices. Any space-like plane in $\AdS^3$ is isometric to the hyperbolic plane $\HH^2$.  Therefore similar to the setting of hyperbolic $3$-space, the path metric on the surface of $P$ is a complete hyperbolic metric on the $N$-times punctured sphere $\Sigma_{0,N}$ determining a point in the Teichm\"uller space $\Teich_{0,N}$, again called the \emph{induced metric}. We show the following result:

\begin{theorem}\label{thm:main-AdS-metrics}
The map $\Phi: \overline{\AdSPoly}_N \to \Teich_{0, N}$, taking a convex ideal polyhedron $P$ in $\AdS^3$ to the induced metric on $\partial P$, is a diffeomorphism.\end{theorem}

\noindent The (weaker) local version of this theorem is a crucial ingredient in proving Theorem~\ref{thm:main-AdS-angles}.

Before continuing on to half-pipe geometry and the cylinder, let us make two remarks about potential generalizations of Theorems~\ref{thm:main-AdS-angles} and~\ref{thm:main-AdS-metrics}.

\begin{remark}[Hyperideal polyhedra]
In the proofs of Theorems~\ref{thm:main-AdS-angles} and~\ref{thm:main-AdS-metrics}, many
of our techniques should apply in the setting of \emph{hyperideal} polyhedra, i.e. polyhedra whose vertices lie outside of the hyperboloid, but all of whose edges pass through the hyperboloid. We believe that similar parameterization statements may hold in this setting.
\end{remark}

\begin{remark}[Relationship with the bending conjecture]
The statements of Theorems~\ref{thm:main-AdS-angles} and~\ref{thm:main-AdS-metrics} bear close resemblance to a conjecture of Mess~\cite{mes_lor} in the setting of globally hyperbolic Cauchy compact AdS space-times. Mess conjectured, by analogy to a related conjecture of Thurston in the setting of quasifuchsian groups, that such a spacetime should be determined uniquely by the bending data or by the induced metric on the boundary of the convex core inside the spacetime. There are existence results known in both cases, due to Bonsante--Schlenker~\cite{bon_fix} and Diallo~\cite{diallo2013} respectively, but no uniqueness or parameterization statement is known in this setting. Ultimately, Theorems~\ref{thm:main-AdS-angles} and~\ref{thm:main-AdS-metrics} on the one hand and Mess's conjecture on the other hand boil down to understanding the connection between the geometry of a subset of $\partial_\infty \AdS^3$ and the geometry of its convex hull in $\AdS^3$. It is natural to ask whether Mess's conjecture and our theorems on ideal polyhedra might naturally coexist as part of some larger universal theory relating the geometry of a convex spacetime in $\AdS^3$ to its asymptotic geometry at the ideal boundary.
\end{remark}

\subsection{A parameterization of ideal polyhedron in $\HP^3$}\label{sec:HP-param}

Half-pipe (HP) geometry is a transitional geometry lying at the intersection of hyperbolic and anti-de Sitter geometry. Intuitively, it may be thought of as the normal bundle of a codimension one hyperbolic plane inside of either hyperbolic space or anti-de Sitter space. In \cite{dan_age, dan_ide}, the first named author constructs paths of three-dimensional projective structures on certain manifolds which transition from hyperbolic geometry to AdS geometry passing through an HP structure. In our setting, it is informative to imagine families of polyhedra in projective space whose vertices lie on a quadric surface evolving from the sphere to the hyperboloid passing through the cylinder. Indeed, the notion of transition is also useful for proving several key statements needed along the way to the main theorems.

Half-pipe geometry is a homogeneous $(G,X)$--geometry. The projective model $X = \HP^3$ for half-pipe space is simply the solid cylinder $x_1^2 + x_2^2 < 1$ in the affine $x_1$-$x_2$-$x_3$ coordinate chart $\RR^3$.  There is a natural projection $\varpi: \HP^3 \to \HH^2$, seen, in this model, as the projection of the solid cylinder to the disk. The projection is equivariant taking projective transformations which preserve the cylinder to isometries of the hyperbolic plane. The projection also extends to take the ideal boundary $\partial_\infty\HP^3 = C$ to the ideal boundary $\partial_\infty\HH^2$ of the hyperbolic plane. The structure group $G$ is the codimension one subgroup of all projective transformations preserving the cylinder which preserves a certain length function along the fibers of this projection. By pullback, the projection $\varpi$ determines a metric on $\HP^3$ which is degenerate along the fiber direction. In this metric, all non-degenerate $2$-planes are isometric to the hyperbolic plane. 

Let $P$ be a convex ideal polyhedron in $\HP^3$ with $N$ vertices. That $P$ is ideal means that the closure of $P$ in $\RP^3$ is a polyhedron contained in $\HP^3 \cup \partial_\infty \HP^3$ whose intersection with $\partial_\infty \HP^3$ is precisely its vertices. Since $\HP^3$ is contained in an affine chart, the notion of convexity is defined to be the same as in affine space. Then the $N$ vertices project to $N$ distinct points on the ideal boundary of the hyperbolic plane (else one of the edges of $P$ would be contained in $\partial_\infty\HP^3$, which we do not allow).
Therefore $P$ determines an ideal polygon $p = \varpi(P)$ in the hyperbolic plane. Further, all facets of an ideal polyhedron in $\HP^3$
are non-degenerate; in particular the faces of $P$ are transverse to the fibers of $\varpi$. By equipping $\HP^3$ with an orientation of the fiber direction, we may sort the faces of $P$ into two types, those for which the outward pointing fiber direction is positive, and those for which it is negative. We call such faces \emph{positive} or \emph{negative}, respectively. The positive faces form a disk (a bent polygon) as do the negative faces. The edges of $P$ which separate a positive face from a negative face form a Hamiltonian cycle in the $1$--skeleton of $P$, again called the \emph{equator}. As in the AdS setting, we let $\HPPoly = \HPPoly_N$ denote the space of all marked non-degenerate convex ideal polyhedra in $\HP^3$ with $N$ vertices, up to orientation preserving and fiber-orientation preserving transfomations. Again, the boundary of each ideal polyhedron $P$ is naturally identified with $\Sigma_{0,N}$ via the (isotopy class of) map taking each ideal vertex to the corresponding puncture and the equator to~$\gamma$. Under this identification, the union of the positive faces (resp. the union of the negative faces) is identified with the top (resp. bottom) disk of $\Sigma_{0,N}$. Via the marking, any ideal polyhedron $P$ realizes the edges of a graph $\Gamma \in\Graph(N, \gamma)$ as a collection of geodesic lines either on the surface of or inside of $P$.

The angle measure between two non-degenerate planes in $\HP^3$ can be defined in terms of the length function on the fibers. Alternatively, one should think of a non-degenerate plane in $\HP^3$ as an infinitesimal deformation of some fixed central hyperbolic plane in $\HH^3$ or $\AdS^3$. As such, the angle between two intersecting planes in $\HP^3$ should be thought of as an infinitesimal version of the standard angle measure in $\HH^3$ or $\AdS^3$. As in the AdS setting, we must distinguish between two types of dihedral angles: two non-degenerate half-planes meeting along a non-degenerate edge $e$ either lie on opposite sides of or the same side of the degenerate plane (which is the union of all degenerate lines) passing through $e$.
As in the AdS setting, we take the convention that the dihedral angles along the equator of a convex ideal polyhedron $P$ are negative, while the dihedral angles along the other edges are positive. Let $\Gamma$ be the $1$--skeleton of $P$ with $\gamma$ subgraph corresponding to the equator. Let $P^*$ be the Poincar\'e dual of $P$. A simple argument in HP geometry (Section~\ref{sec:proof-main-HP}) shows that the function $\theta$ assigning to each edge $e^*$ of $P^*$ the dihedral angle at the corresponding edge $e$ of $P$ satisfies the same three conditions (i), (ii), and (iii) of the previous section; in other words $\theta \in \Angles$. Define $\HPPoly_{\Gamma}$ to be the space of ideal polyhedra in $\HP^3$ with $1$--skeleton identified with $\Gamma \in\Graph(N, \gamma)$ and let $\PsiHP_{\Gamma}: \HPPoly_{\Gamma} \to \Angles_{\Gamma}$ be the map assigning to an ideal polyhedron its dihedral angles. Then all of the maps $\PsiHP_{\Gamma}: \HPPoly_{\Gamma} \to \Angles_{\Gamma}$ may be, again, stitched together into one. We show:

\begin{theorem}
\label{thm:main-HP}
The map $\PsiHP: \HPPoly \to \Angles$, defined by $\PsiHP(P) = \PsiHP_{\Gamma}(P)$, if $P \in \HPPoly_{\Gamma}$, is a homeomorphism.
\end{theorem}

\noindent The equivalence of conditions (C) and (H) in Theorem~\ref{thm:main} follows from Theorem~\ref{thm:main-HP} and Theorem~\ref{thm:main-AdS-angles}. Note that there is no direct analogue of Theorem~\ref{thm:main-AdS-metrics} in the half-pipe setting. Indeed the induced metric on a ideal polyhedron in $\HP^3$ is exactly the double of the ideal polygon $\varpi(P)$ and the space of such doubles is a half-dimensional subset of $\Teich_{0,N}$. Intuitively, the induced metric does not determine $P$ because, as a polyhedron in $\HH^3$ (or $\AdS^3$) collapses onto a plane, the induced metric only changes to second order: the path metric on a plane bent by angle $\theta$ differs from the ambient metric only to second order in~$\theta$.

\subsection{Strategy of the proofs and organization}

There is a natural relationship between bending in $\AdS^3$ and earthquakes on hyperbolic surfaces. We describe this relationship, in our context of interest, in Section~\ref{sec:models}. Here is a synopsis. Via the product structure on the ideal boundary $\partial_\infty \AdS^3 \cong \RP^1 \times \RP^1$, an ideal polyhedron $P \in \AdSPoly_N$ is determined by two ideal polygons $p_L$ and $p_R$ in the hyperbolic plane, each with $N$ labeled vertices (see Section \ref{sec:ideal_poly}). The two metrics $m_L, m_R \in \Teich_{0,N}$ obtained by doubling $p_L$ and $p_R$ respectively are called the \emph{left  metric} and \emph{right metric} respectively. Given weights $\theta$ on a graph $\Gamma \in \Graph(\Sigma_{0,N}, \gamma)$, the pair $p_L, p_R$ determine an ideal polyhedron $P$ with bending data $\theta$ if and only if the left and right metrics satisfy:
\begin{equation}\label{eqn:earthquake-diagram-intro}
m_R = E_{2\theta} m_L,
\end{equation} where $E_\theta$ is the shear map defined by shearing a surface along the edges of $\Gamma$ according to the weights given by $\theta$ (where a positive weight means shear to the left, and a negative weight means shear to the right). Directly solving for $p_L$ and $p_R$ given $\theta$ is very difficult. However, the infinitesimal version of this problem is more tractable; this is the relevant problem in the setting of half-pipe geometry.

An ideal polyhedron $P \in \HPPoly_N$ is determined by an $N$-sided ideal polygon $p$ in the hyperbolic plane and an infinitesimal deformation $V$ of $p$ (see Section~\ref{sec:models}). Doubling yields an element $m$ of the Teichm\"uller space $\Teich_{0,N}$ and an infinitesimal deformation $W$ of $m$ which is tangent to the sub-space of doubled ideal polygons. The data $p, V$ determine an ideal polyhedron $P \in \HPPoly$ with bending data $\theta$ if and only if the infinitesimal deformation $W$ is obtained by infinitesimally shearing $m$ along the edges of $\Gamma$ according to the weights $\theta$. In Section~\ref{sec:lengthfunctions}, we show how to solve for the polygon $p$ given $\theta \in \Angles_{\Gamma}$ by minimizing an associated length function. In Section~\ref{sec:proof-main-HP}, we apply the results of Section~\ref{sec:lengthfunctions} to directly prove Theorem~\ref{thm:main-HP}, that $\PsiHP$ is a homeomorphism, after first proving: 

\begin{prop}
\label{prop:imageinA-HP}
The map $\PsiHP_{\Gamma}$ taking an ideal polyhedron $P \in \HPPoly_{\Gamma}$ to its dihedral angles $\theta$ has image in $\Angles_{\Gamma}$. In other words, $\theta$ satisfies conditions (i), (ii), and (iii) of Section~\ref{sec:intro-AdS-param}.
\end{prop} 
\noindent The proof of this proposition is a simple computation in half-pipe geometry, which uses (among other things) an infinitesimal version of the Gauss--Bonnet theorem for polygons.

In the AdS setting constructing inverses for the maps $\PsiAdS$ and $\Phi$ is too difficult, so we proceed in the usual next-best way: we prove each map is a proper, local homeomorphism, and then argue via topology. Because Teichm\"uller space $\Teich_{0,N}$ is a ball and because $\overline{\AdSPoly}_N$ is connected and has dimension equal to that of $\Teich_{0,N}$ (Proposition~\ref{topol_poly}), Theorem~\ref{thm:main-AdS-metrics} is implied by the following two statements.

\begin{lemma}\label{lem:Phi-proper}
The map $\Phi: \overline{\AdSPoly}_N \to \Teich_{0,N}$ is proper.
\end{lemma}
\begin{lemma}
\label{lem:Phi-rigidity}
The map $\Phi: \overline{\AdSPoly}_N \to \Teich_{0,N}$ is a local immersion.
\end{lemma}

Lemma~\ref{lem:Phi-proper} is proved in Section~\ref{sec:properness} by directly studying the effect of degeneration of the left and right metrics $m_L, m_R$ of $P$ on the induced metric $\Phi(P)$ via Equation~\eqref{eqn:earthquake-diagram-intro}.
Lemma~\ref{lem:Phi-rigidity} is deduced in Section~\ref{sec:rigidity} from a similar rigidity statement in the setting of convex Euclidean polyhedra using an \emph{infinitesimal Pogorelov map}, which is a tool that translates infinitesimal rigidity questions form one constant curvature geometry to another.

Next, to prove Theorem~\ref{thm:main-AdS-angles}, we need the relevant local parameterization and properness statements in the setting of dihedral angles.  Note that in the following lemmas, we consider each $\PsiAdS_\Gamma$ as having image in $\RR^E$, where again $E$ is the set of edges of the graph $\Gamma \in\Graph(N, \gamma)$. The first lemma is a properness statement for $\PsiAdS$.

\begin{lemma}
\label{lem:PsiAdS-proper}
Consider a sequence $P_n \in \AdSPoly_{\Gamma}$ going to infinity in $\AdSPoly$ such that the dihedral angles $\theta_n = \PsiAdS_{\Gamma}(P_n)$ converge to $\theta_\infty \in \RR^E$. Then $\theta_\infty$ fails to satisfy condition (iii) of Section~\ref{sec:intro-AdS-param}.
\end{lemma}

\noindent Lemma~\ref{lem:PsiAdS-proper} is proven in Section~\ref{sec:properness} together with Lemma~\ref{lem:Phi-proper}. In the next lemma, we assume $\Gamma$ is a triangulation (i.e. maximal) and extend the definition of $\Psi_\Gamma$ to all of $\AdSPoly$. Indeed, for $P \in \AdSPoly$, each ideal triangle of $\Gamma$ is realized as a totally geodesic ideal triangle in $P$. Therefore, the punctured sphere $\Sigma_{0,N}$ maps into $P$ as a bent (but possibly not convex) totally geodesic surface with $1$--skeleton $\Gamma$ and we may measure the dihedral angles (with sign) along the edges. 

\begin{lemma}\label{lem:PsiAdS-rigidity}
Assume $\Gamma$ is a triangulation of $\Sigma_{0,N}$, with $E$ denoting the set of $3N - 6$ edges of $\Gamma$. 
If the $1$--skeleton of $P \in \AdSPoly$ is a subgraph of~$\Gamma$, then $\PsiAdS_{\Gamma}: \AdSPoly \to \RR^E$ is a local immersion near $P$. \end{lemma}
\noindent Lemma~\ref{lem:PsiAdS-rigidity} is obtained as a corollary of Lemma~\ref{lem:Phi-rigidity} via a certain duality between metric data and bending data derived from the natural pseudo-complex structure on $\AdSPoly$. See Section~\ref{sec:pseudo} and Section~\ref{sec:rigidity}. 

The next ingredient for Theorem~\ref{thm:main-AdS-angles} is:
\begin{prop}\label{prop:imageinA-AdS}
The map $\PsiAdS_{\Gamma}$ taking an ideal polyhedron $P \in \AdSPoly_{\Gamma}$ to its dihedral angles $\theta$ has image in $\Angles_{\Gamma}$. 
\end{prop}

\noindent The content of this proposition is that $\PsiAdS_\Gamma(P)$ satisfies condition~(iii) of Section~\ref{sec:intro-AdS-param} (conditions (i) and (ii) are automatic).
This will be proven directly in Section~\ref{sec:necess} by a computation in $\AdS$ geometry. See Appendix~\ref{sec:transitional-proof} for an alternative indirect proof using transitional geometry.

In Section~\ref{sec:topology}, we explain why Lemmas~\ref{lem:PsiAdS-proper} and~\ref{lem:PsiAdS-rigidity}, and Proposition~\ref{prop:imageinA-AdS} imply that $\PsiAdS$ is a covering onto $\Angles$. We then argue that $\Angles$ is connected and simply connected when $N \geq 6$ using Theorem~\ref{thm:main-HP}, and we prove Theorem~\ref{thm:main-AdS-angles} (treating the cases $N=4,5$ separately). We also deduce Theorems \ref{thm:main} from Theorem  \ref{thm:main-AdS-angles}, \ref{thm:main-HP} and Rivin's theorem.

\vspace{0.1in}
{\noindent \bf Acknowledgements}
Some of this work was completed while we were in residence together at the 2012 special program on Geometry and analysis of surface group representations at the Institut Henri Poincar\'e; we are grateful for the opportunity to work in such a stimulating environment.  Our collaboration was greatly facilitated by support from the GEAR network (U.S. National Science Foundation grants DMS 1107452, 1107263, 1107367 ``RNMS: GEometric structures And Representation varieties'').

%
%

\section{Hyperbolic, anti-de Sitter, and half-pipe geometry in dimension 3}
\label{sec:models}

This section is dedicated to the description of the three-dimensional geometries of interest in this paper, and to the relationship between these geometries. We prove a number of basic but fundamental theorems, some of which have not previously appeared in the literature as stated. Of central importance is the interpretation of bending data in these geometries in terms of shearing deformations in the hyperbolic plane (Theorem~\ref{thm:earthquakes} and~\ref{thm:HP-bend}).

In \cite{dan_age}, the first named author constructs a family of model geometries in projective space that transitions from hyperbolic geometry to anti-de Sitter geometry, passing though half-pipe geometry. We review the dimension-three version of this construction here. Each model geometry $\XX = \XX(\BB)$ is associated to a real two-dimensional commutative algebra~$\BB$.

Let $\BB = \mathbb R + \mathbb R \kappa$ be the real two-dimensional, commutative algebra generated by a non-real element $\kappa$ with $\kappa^2 \in \RR$. As a vector space $\BB$ is spanned by $1$ and $\kappa$. There is a conjugation action: 
$ \overline{(a + b \kappa)} := a - b \kappa, $
which defines a square-norm
$$ |a + b\kappa|^2 := (a + b\kappa)\overline{(a + b\kappa)} = a^2 -b^2 \kappa^2 \ \in \ \mathbb R.$$
Note that $|\cdot|^2$ may not be positive definite. We refer to $a$ as the \emph{real part}  and $b$ as the \emph{imaginary part} of $a + b\kappa$. If $\kappa^2 = -1$, then our algebra $\mathcal B = \CC$ is just the complex numbers, and in this case we use the letter~$i$ in place of~$\kappa$, as usual. If $\kappa^2 = +1$, then $\mathcal B$ is the \emph{pseudo-complex (or Lorentz) numbers} and we use the letter~$\tau$ in place of~$\kappa$. In the case $\kappa^2 = 0$, we use the letter $\sigma$ in place of $\kappa$. In this case $\mathcal B = \Rsigma$ is isomorphic to the tangent bundle of the real numbers. Note that if $\kappa^2 < 0$, then $\mathcal B \cong \CC$, and if $\kappa^2 > 0$ then $\mathcal B \cong \Rtau$.

Now consider the $2\times2$ matrices $M_2(\BB)$. 
Let $$\Herm(2, \BB) = \{ A \in M_2(\BB) : A^* = A\}$$ denote the $2\times 2$ Hermitian matrices, where $A^*$ is the conjugate transpose of $A$.  As a real vector space, $\Herm(2,\BB) \cong \mathbb R^4$. We define the following (real) inner product on $\Herm(2,\BB)$:
$$ \left\langle \bminimatrix{a}{z}{\bar{z}}{d}, \bminimatrix{e}{w}{\bar{w}}{h} \right\rangle = -\frac{1}{2}tr\left( \bminimatrix{a}{z}{\bar{z}}{d} \bminimatrix{h}{-w}{-\bar{w}}{e}\right).$$
We will use the coordinates on $\Herm(2,\BB)$ given by 
\begin{align} \label{coordinates-on-Herm}
X &= \bminimatrix{x_4+x_1}{x_2 - x_3 \kappa}{x_2 + x_3\kappa}{x_4 - x_1}.
\end{align} In these coordinates, we have that
$$\langle X, X\rangle = -\text{det}(X) = x_1^2 + x_2^2 - \kappa^2 x_3^2 - x_4^2,$$
and we see that the signature of the inner product is $(3,1)$ if $\kappa^2 < 0$, or $(2,2)$ if $\kappa^2 > 0$.

The coordinates above identify $\Herm(2,\BB)$ with $\RR^4$. Therefore we may identify the real projective space $\RP^3$ with the non-zero elements of $\Herm(2,\BB)$, considered up to multiplication by a real number. We define the region $\XX$ inside $\RP^3$ as the negative lines with respect to $\langle \cdot, \cdot\rangle$:
$$\XX = \left\{ X \in \Herm(2,\BB) : \langle X, X\rangle < 0 \right\} / \RR^*.$$
Note that in the affine chart $x_4 = 1$, our space $\XX$ is the standard round ball if $\kappa = i$, the standard solid hyperboloid if $\kappa = \tau$, or the standard solid cylinder if $\kappa = \sigma$.

Next, define the group $\PGL^+(2,\mathcal B)$ to be the $2\times 2$ matrices $A$, with coefficients in~$\BB$, such that $|\det(A)|^2 > 0$, up to the equivalence $A \sim \lambda A$ for any $\lambda \in \mathcal B^{\times}$.
The group $\PGL^+(2, \BB)$ acts on $\XX$ by orientation preserving projective linear transformations as follows. Given $A \in \PGL^+(2,\BB)$ and $X \in \XX$:
$$ A \cdot X := A X A^*.$$
\begin{remark}\label{rem:plane}
The matrices with real entries determine a copy of $\PSL(2,\RR)$ inside of $\PGL^+(2, \BB)$, which preserves the set $\plane$ of negative lines in the $x_1$-$x_2$-$x_4$ plane (in the coordinates above). The subspace $\plane$ of $\XX$ is naturally a copy of the projective model of the hyperbolic plane. We think of $\plane$ as a common copy of $\HH^2$ contained in every model space $\XX = \XX(\BB)$ independent of the choice of $\kappa^2$.
\end{remark}

Note that if $\BB = \CC$, then $\PGL^+(2,\BB) = \PSL(2,\CC)$ and $\XX$ identifies with the usual projective model for \emph{hyperbolic space} $\XX = \HH^3$. In this case, the action above is the usual action by orientation preserving isometries of $\HH^3$, and gives the familiar isomorphism $\PSL(2,\CC) \cong \PSO(3,1)$,

If $\BB = \Rtau$, with $\tau^2 = +1$, then $\XX$ identifies with the usual projective model for \emph{anti-de Sitter space} $\XX = \AdS^3$. Anti-de Sitter geometry is a Lorentzian analogue of hyperbolic geometry. The inner product $\langle \cdot, \cdot \rangle$ determines a metric on $\XX$, defined up to scale. We choose the metric with constant curvature $-1$. Note that the metric on $\AdS^3$ has signature $(2,1)$, so tangent vectors are partitioned into three types: \emph{space-like}, \emph{time-like}, or \emph{light-like}, according to whether the inner product is positive, negative, or null, respectively. In any given tangent space, the light-like vectors form a cone that partitions the time-like vectors into two components. Thus, locally there is a continuous map assigning the name \emph{future pointing} or \emph{past pointing} to time-like vectors. The space $\AdS^3$ is \emph{time-orientable}, meaning that the labeling of time-like vectors as future or past may be done consistently over the entire manifold. The action of $\PGL^+(2,\Rtau)$ on $\AdS^3$ is by isometries, thus giving an embedding $\PGL^+(2,\Rtau) \hookrightarrow \PSO(2,2)$. In fact, $\PGL^+(2,\Rtau)$ has two components, distinguished by whether or not the action on $\AdS^3$ preserves time-orientation, and the map is an isomorphism.

Lastly, we discuss the case $\BB = \Rsigma$, with $\sigma^2 = 0$. In this case, $\XX = \HP^3$ is the projective model for \emph{half-pipe geometry} (HP), defined in \cite{dan_age} for the purpose of describing a geometric transition going from hyperbolic to AdS structures. 
The algebra $\mathbb R + \mathbb R \sigma$ should be thought of as the tangent bundle of $\mathbb R$: Letting $x$ be the standard coordinate function on $\mathbb R$, we think of $a + b \sigma$ as a path based at $a$ with tangent $b \frac{\partial}{\partial x}$. More appropriately, one should think of $\Rsigma$ as the bundle of imaginary directions in $\CC$ (resp. $\Rtau$) restricted to the subspace $\RR$. See Section~\ref{subsec:polyhedra-HP}.

\begin{Remark}
In each case, the orientation reversing isometries are also described by $\PGL^+(2, \BB)$ acting by $X \mapsto A \overline{X} A^*$.
\end{Remark}

Although, we focus on dimension three, there are projective models for these geometries in all dimensions. Generally, the $n$-dimensional hyperbolic space $\HH^n$ (resp. the $n$-dimensional anti-de Sitter space $\AdS^n$) may be identified with the space of negative lines in $\RP^n$ with respect to a quadratic form of signature $(n,1)$ (resp. of signature $(n-1,2)$); the isometry group is the projective orthogonal group with respect to this quadratic form, isomorphic to $\PO(n,1)$ (resp. $\PO(n-1,1)$). The $n$-dimensional half-pipe space $\HP^n$ identifies with the space of negative lines with respect to a degenerate quadratic form with $n-1$ positive eigenvalues, one negative eigenvalue, and one zero eigenvalue. The structure group, as in the three-dimensional case, is a codimension one subgroup of all projective transformations preserving this set. See Section~\ref{s:HP3-geometry}.

\vspace{0.2cm}

{\noindent \bf The ideal boundary.} The ideal boundary $\partial_\infty\XX$ is the boundary of the region $\XX$ in $\RP^3$. It is given by the null lines in $\Herm(2,\BB)$ with respect to $\langle \cdot, \cdot \rangle$. Thus $$\partial_\infty\XX = \left\{ X \in \Herm(2,\BB) : \det(X) = 0, X \neq 0\right\}/\RR^*$$ can be thought of as the $2\times 2$ Hermitian matrices of rank one. We now give a useful description of $\partial_\infty \XX$ that generalizes the identification $\partial_\infty \HH^3 = \CP^1$. 

Any rank one Hermitian matrix $X$ can be decomposed (up to $\pm$) as 
\begin{equation} \label{eqn:decomposition}
X= \pm v v^*,
\end{equation} where $v \in \BB^2$ is a two-dimensional column vector with entries in $\BB$, unique up to multiplication by $\lambda \in \BB$ with $|\lambda|^2 = 1$ (and $v^*$ denotes the transpose conjugate). 
This gives the identification $$\partial_\infty \mathbb X \cong \mathbb P^1 \BB = \left\{ v \in \BB^2 : v v^* \neq 0 \right\} / \sim,$$ where $v \sim v \lambda \text{ for } \lambda \in \BB^{\times}.$  The action of $\PGL^+(2, \BB)$ on $\mathbb P^1 \BB$ by matrix multiplication extends the action of $\PGL^+(2,\BB)$ on $\XX$ described above. We note also that the metric on $\XX$ determines a compatible conformal structure on $\partial_{\infty} \XX = \mathbb P^1 \BB$. Restricted to $\BB \subset \mathbb P^1 \BB$, this conformal structure is exactly the conformal structure induced by the square-norm $| \cdot |^2$. In particular, it is Euclidean if $\kappa^2 < 0$, Lorentzian if $\kappa^2 > 0$, or degenerate if $\kappa^2 = 0$.

We use the square-bracket notation $\btwovector{x}{y}$ to denote the equivalence class in $\mathbb P^1 \BB$ of $\begin{pmatrix}x \\ y \end{pmatrix} \in \BB^2$. Similarly, a $2\times 2$ square-bracket matrix $\bminimatrix{a}{b}{c}{d}$ denotes the equivalence class in $\PGL^+(2,\BB)$ of the matrix $\minimatrix{a}{b}{c}{d} \in \GL^+(2,\BB)$. Throughout, we will identify $\BB$ with its image under the injection $\BB \hookrightarrow \mathbb P^1 \BB$ given by $z \mapsto \btwovector{z}{1}$.

\begin{Remark}
In the case $\kappa^2 \geq 0$, the condition $v v^* \neq 0$ in the definition of $\mathbb P^1 \BB$ is \emph{not} equivalent to the condition $v \neq 0$, because $\BB$ has zero divisors.\end{Remark}

The inclusion $\RR \hookrightarrow \BB$ induces an inclusion $\RP^1 \hookrightarrow \mathbb P^1 \BB$. This copy of $\RP^1$ is precisely the ideal boundary of the common hyperbolic plane $\plane$ contained in all model spaces $\XX$ (independent of the choice of $\kappa^2$).

Recall that a subset $P$ of projective space is called \emph{convex} if $P$ is contained in an affine chart and is convex in that affine chart.
In the notation introduced here, the fundamental objects of this article are defined as follows:

\begin{defi}
A \emph{convex ideal polyhedron} in $\XX$ is a convex polyhedron $P$ in projective space such that the vertices of $P$ lie in $\partial_\infty \XX$ and the rest of $P$ lies in $\XX$.
\end{defi}

\noindent An \emph{ideal triangle} in $\XX$ is a convex ideal polyhedron with three vertices.
An \emph{ideal simplex} or \emph{ideal tetrahedron} is a convex ideal polyhedron with four vertices. Ideal simplices and their moduli will play an important role in this article. We review some of the basic theory, referring the reader to~\cite{dan_ide} for a more detailed account.

Let $Z_1, Z_2, Z_3, Z_4 \in \Herm(2,\BB)$ have rank one, and let $z_1, z_2, z_3, z_4$ denote the corresponding elements of $\mathbb P^1 \BB$. Assume that $Z_1, Z_2, Z_3$ determine an ideal triangle in $\XX$. There is a unique $A \in \PGL^+(2,\BB)$ such that $A z_1 = \infty := \btwovector{1}{0},  Az_2 = 0:= \btwovector{0}{1}$, and $A z_3 = 1:= \btwovector{1}{1}$. Then $$(z_1,z_2; z_3,z_4) := A z_4$$ is an invariant of the ordered ideal points $z_1,\ldots,z_4$, which will be referred to as the \emph{cross ratio} of the four points, since it generalizes the usual cross ratio in $\CP^1$. 
It is straighforward to check that $z_1,z_2,z_3,z_4$ define an ideal tetrahedron in $\mathbb X$ if and only if $z = (z_1,z_2;z_3,z_4)$ (is defined and) lies in $\mathcal B \subset \mathbb P^1 \mathcal B$ and satisfies:
\begin{equation}
|z|^2, |1-z|^2 > 0.
\label{spacelike}
\end{equation}
In this case $z$ is called the \emph{shape parameter} of the ideal tetrahedron (with ordered vertices $z_1, z_2, z_3, z_4$). Using the language of Lorentzian geometry, we say that $z$ and $z-1$, as in \eqref{spacelike}, are \emph{space-like}. In fact, all facets of an ideal tetrahedron are space-like and totally geodesic with respect to the metric induced by $\langle \cdot, \cdot \rangle$ on $\mathbb X$. The shape parameter $z$ is a natural geometric quantity associated to the edge $e = z_1z_2$ of the tetrahedron in the following sense, described in Thurston's notes \cite[\S 4]{thu_the} in the hyperbolic case.  Change coordinates (using an element of $\PGL^+(2,\BB)$) so that $z_1 = \infty$, and $z_2 = 0$. Then the subgroup $G_e$ of $\PGL^+(2,\BB)$ that preserves $e$ is given by $$G_e = \left\{ A = \bminimatrix{\lambda}{0}{0}{1}: \lambda \in \BB, |\lambda|^2 > 0 \right\}.$$ The number $\lambda = \lambda(A)$ associated to $A \in G_e$ is called the \emph{exponential $\BB$-length} and generalizes the exponential complex translation length of a loxodromic element of $\PSL(2,\CC)$. Let $A \in G_e$ be the unique element so that $A z_3 = z_4$. Then the shape parameter is just the exponential $\BB$-length of $A$: $z = \lambda(A)$. 

There are shape parameters associated to the other edges as well. We may calculate them as follows. Let $\pi$ be any \emph{even} permutation of $\{1,2,3,4\}$, which corresponds to an orientation preserving diffeomorphism of the standard simplex. Then $(z_{\pi(1)}, z_{\pi(2)}; z_{\pi(3)}, z_{\pi(4)})$ is the shape parameter associated to the edge $e' = z_{\pi(1)} z_{\pi(2)}$. This definition a priori depends on the orientation of the edge $e'$. However, one easily checks that $(z_2,z_1;z_4,z_3) = (z_1,z_2; z_3,z_4).$ Figure~\ref{fig:shape-params} summarizes the relationship between the shape parameters of the six edges of an ideal tetrahedron, familiar from the hyperbolic setting.

\begin{figure}[h]
{\centering
\def\svgwidth{1.5in}
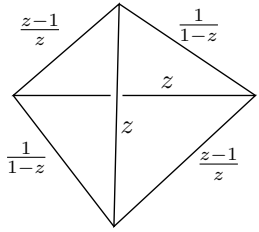
}
\caption{\label{fig:shape-params} The shape parameters corresponding to the six edges of an ideal tetrahedron.}
\end{figure}

\subsection{Hyperbolic geometry in dimension three}
Let $\kappa^2 = -1$, so that $\mathcal B = \mathbb C$ is the complex numbers. In this case, the inner product $\langle \cdot, \cdot \rangle$ on $\Herm(2,\mathbb C)$ is of type $(3,1)$ and $\mathbb X$ is the unit ball in the affine chart $x_4 = 1$, known as the projective model for $\mathbb H^3$. A basic understanding of hyperbolic geometry, although not the main setting of interest, is very important for many of the arguments in this article. 
We will often use intuition from the hyperbolic setting as a guide, and so we assume the reader has a basic level of familiarity. Let us recall some basic facts here and present an important theorem, whose analogue in the AdS setting will be crucial.

The ideal boundary $\partial_\infty \HH^3$ identifies with $\mathbb P^1 \BB = \CP^1$. Since the ball is strictly convex, any $N$ distinct points $z_1,\ldots, z_N$ determine an ideal polyhedron $P$ in $\HH^3$. In the case $N = 4$, the ideal simplex $P$ is determined by the shape parameter $z = (z_1, z_2; z_3, z_4) \in \CC$. Indeed, Condition~\eqref{spacelike} gives the well-known fact that the shape parameter $z$ may take any value in $\mathbb C\setminus \{0,1\}$.  
Consider the two faces $T = \Delta z_1 z_2 z_3$ and $T' = \Delta z_2 z_1 z_4$ of $P$, each oriented compatibly with the outward pointing normal, meeting along the edge $e = z_1 z_2$. Then, writing $z = e^{s + i \theta}$, the quantity $s$ is precisely the amount of \emph{shear} along $e$ between $T$ and $T'$, while $\theta$ is precisely the dihedral angle at $e$. 

An infinitesimal deformation of an ideal polyhedron $P$ is given by a choice $V = (V_1, \ldots, V_N)$ of tangent vectors to $\CP^1$ at each of the vertices $z_1, \ldots, z_N$ of $P$. Such a deformation is considered trivial if $V_1, \ldots, V_N$ are the restriction of a global Killing field on $\HH^3 \cup \CP^1$ to the vertices $z_1, \ldots, z_n$. 
If necessary, augment the $1$--skeleton of $P$ so that it is an ideal triangulation $\Gamma$ of the surface of $P$. Then the map $z_\Gamma$, taking an ideal polyhedron $P$ to the collection of $3N-6$ cross ratios associated to the edges of $\Gamma$, is holomorphic and the following holds:

\begin{theorem}\label{thm:H3-duality}
An ideal polyhedron $P$ is infinitesimally rigid with respect to the induced metric if and only if $P$ is infinitesimally rigid with respect to the dihedral angles.
\end{theorem}

\begin{proof}
Since the induced metric is determined entirely by the shear coordinates with respect to $\Gamma$, we have that the infinitesimal deformation $V$ does not change the induced metric to first order if and only if  $d\log z_\Gamma (V)$ is pure imaginary. On the other hand, $V$ does not change the dihedral angles to first order if\ and only if $d\log z_\Gamma (V)$ is real. Therefore $V$ does not change the induced metric if and only if $i V$ does not change the dihedral angles.
\end{proof}

\begin{Remark}
Theorem~\ref{thm:H3-duality} is a simpler version of Bonahon's argument~\cite{bon_she} that a hyperbolic three-manifold is rigid with respect to the metric data on the boundary of the convex core if and only if it is rigid with respect to bending data on the boundary of the convex core. In this setting of polyhedra, Bonahon's shear-bend cocycle is replaced by a finite graph $\Gamma$ with edges labeled by the relevant shape parameters $z$ (or $\log z$).
\end{Remark}

\subsection{Anti-de Sitter geometry in dimension three}\label{ads_background}

Let $\mathcal{B}$ be the real algebra generated by an element $\tau$, with $\tau^2 = +1$, which defines $\mathbb X = \AdS^3$, the anti-de Sitter space. Let us discuss some important properties of the algebra $\mathcal B = \mathbb R + \mathbb R \tau$, known as the pseudo-complex numbers.

\vspace{0.2cm}

{\noindent \bf The algebra $\BB = \mathbb R + \mathbb R \tau$ of pseudo-complex numbers.}

First, note that $\mathcal{B}$ is not a field as, for example, $(1+\tau)\cdot(1-\tau) = 0.$
The square-norm defined by the conjugation operation $|a+b\tau|^2 = (a + b \tau) \overline{(a + b\tau)} = a^2 - b^2$ comes from the $(1,1)$ Minkowski inner product on $\mathbb R^2$ (with basis  $\{1, \tau\}$). The space-like elements of $\mathcal{B}$ (i.e. square-norm $ > 0$), acting by multiplication on $\mathcal{B}$, form a group and can be thought of as the similarities of the Minkowski plane that fix the origin. Note that if $|a + b\tau|^2 = 0$, then $b=\pm a$, and multiplication by $a + b\tau$ collapses all of $\mathcal{B}$ onto the light-like line spanned by $a + b\tau$.

The elements $\frac{1+\tau}{2}$ and $\frac{1-\tau}{2}$ are two spanning idempotents which annihilate one another: $$\left(\frac{1\pm\tau}{2}\right)^2 = \frac{1\pm\tau}{2}, \ \text{ and } \ \left(\frac{1+\tau}{2}\right) \cdot  \left(\frac{1-\tau}{2}\right) = 0.$$ Thus $\mathcal{B} \cong \mathbb R \oplus \mathbb R$, as $\mathbb R$--algebras, via the isomorphism 
\begin{equation}(\varpi_L, \varpi_R): a\left(\frac{1-\tau}{2}\right) + b \left(\frac{1+\tau}{2}\right) \longmapsto (a,b).
\label{isomorphism}
\end{equation}
Here $\varpi_L$ and $\varpi_R$ are called the left and right projections $\BB \to \RR$. These projections extend to left and right projections $\mathbb P^1 \BB \to \RP^1$ which give the isomorphism $\mathbb P^1 \BB \cong \RP^1 \times \RP^1$. Indeed, $\mathbb P ^1 \mathcal B$ is the \emph{Lorentz compactification} of $\mathcal B = \left\{ \btwovector{x}{1} : \ x \in \mathcal B \right\}$. The added points make up a wedge of circles, so that $\mathbb P^1 \mathcal B$ is topologically a torus. The square-norm $|\cdot|^2$ on $\mathcal B$ induces a flat conformal Lorentzian structure on $\mathbb P^1 \mathcal B$ that is preserved by $\PGL^+(2, \mathcal B)$. We refer to $\PGL^+(2, \mathcal B)$ as the \emph{Lorentz M\"obius transformations}. With its conformal structure $\mathbb P^1 \mathcal B$ is the $(1+1)$-dimensional Einstein universe $\text{Ein}^{1,1}$ (see e.g. \cite{barbot-1,primer-einstein} for more about Einstein space).

The splitting $\BB \cong \RR \oplus \RR$ determines a similar splitting $M_2 (\mathcal B) \cong M_2 \RR \oplus M_2 \RR$ of the algebra of $2 \times 2$ matrices which respects the determinant in the following sense: $\forall A \in M_2 (\mathcal B)$ $$(\varpi_L \det A, \varpi_R \det A) = (\det \varpi_L(A), \det \varpi_R(A)),$$
where, by abuse of notation, $\varpi_L$ and $\varpi_R$ also denote the extended maps $M_2 (\mathcal B) \to M_2 (\R)$. The orientation preserving isometries $\operatorname{Isom}^+ \AdS^3 = \PGL^+(2,\mathcal B)$ correspond to the subgroup of $\PGL(2,\mathbb R) \times \PGL(2,\mathbb R)$ such that the determinant has the same sign in both factors. The identity component of the isometry group (which also preserves time orientation) is given by $\PSL(2,\RR) \times \PSL(2,\RR)$.

Note also that the left and right projections $\varpi_L, \varpi_R: \mathbb P^1 \BB \to \RP^1$ respect the cross ratio: $$(z_1, z_2; z_3, z_4) = \dfrac{1-\tau}{2} (\varpi_L z_1 , \varpi_L z_2; \varpi_L z_3, \varpi_L z_4) + \dfrac{1+\tau}{2} (\varpi_R z_1, \varpi_R z_2; \varpi_R z_3, \varpi_R z_4),$$ where on the right-hand side $(\cdot,\cdot;\cdot,\cdot)$ denotes the usual cross ratio in $\RP^1$.

\subsection{Ideal Polyhedra in $\AdS^3$}\label{sec:ideal_poly}

Consider an ideal polyhedron $P$ in $\AdS^3$ with $N$~vertices $z_1, \ldots, z_N \in \mathbb P^1 \BB$. For each $i = 1, \ldots, N$, let $x_i = \varpi_L(z_i)$ and $y_i = \varpi_R(z_i)$ be the left and right projections of $z_i$. Then, all of the $x_i$ (resp. all of the $y_i$) are distinct. Otherwise, the convex hull of the $z_i$ (in any affine chart) will contain a full segment in the ideal boundary.

\begin{prop}
\label{prop:cyclic-order}
The vertices $z_1, \ldots, z_N \in \mathbb P^1 \BB$ determine an ideal polyhedron $P$ in $\AdS^3$ if and only the left projections $x_1, \ldots, x_N$ and right projections $y_1, \ldots, y_N$ are arranged in the same cyclic order on the circle $\RP^1$.
\end{prop}

\begin{proof}
In general, a closed set $\Omega$ in $\RP^M$ is convex if and only any $M+1$ points of $\Omega$ span a (possibly degenerate) simplex contained in $\Omega$.
Therefore the $z_1, \ldots, z_N$ define an ideal polyhedron if and only if any four vertices $z_{i_1}, z_{i_2}, z_{i_3}, z_{i_4}$ span an ideal simplex. This is true if and only if the cross ratio $z = (z_{i_1}, z_{i_2}; z_{i_3}, z_{i_4})$ is defined and satisfies that $|z|^2, |1-z|^2 > 0$. Since $z = \frac{1-\tau}{2}x + \frac{1+\tau}{2}y$, where $x = (x_{i_1}, x_{i_2}; x_{i_3}, x_{i_4})$ and $y = (y_{i_1}, y_{i_2}; y_{i_3}, y_{i_4})$, we have that $|z|^2 = xy$ and $|1-z|^2 = (1-x)(1-y)$. So $|z|^2, |1-z|^2 > 0$ if and only if $x$ and $y$  have the same sign and $(1-x)$ and $(1-y)$ have the same sign. Hence, $z_{i_1}, z_{i_2}, z_{i_3}, z_{i_4}$ span an ideal simplex if and only if the two four-tuples of vertices $(x_{i_1}, x_{i_2}, x_{i_3}, x_{i_4})$ and $(y_{i_1}, y_{i_2}, y_{i_3}, y_{i_4})$ are arranged in the same cyclic order on $\RP^1$. The proposition follows by considering all subsets of four vertices.
\end{proof}

\noindent We denote by $p_L = \varpi_L(P)$ (resp. $p_R = \varpi_R(P)$) the ideal polygon in the hyperbolic plane with vertices $x_1, \ldots, x_N$ (resp. $y_1,\ldots, y_N$). 

Let us quickly recall the definitions and terminology from Section~\ref{sec:intro-AdS-param}. We fix, once and for all, a time orientation on $\AdS^3$. Since all faces of an ideal polyhedron $P$ are space-like, the outward normal to each face is time-like and points either to the future or to the past. This divides the faces into two groups, the future (or top) faces, and the past (or bottom) faces. The union of the future faces is a bent polygon, as is the union of the past faces. The edges dividing the future faces from the past faces form a Hamiltonian cycle, called the \emph{equator}, in the $1$--skeleton of~$P$. We may project $P$ combinatorially to the left and right ideal polygons $p_L$ and $p_R$ respectively. Each face of $P$ is isometric to an ideal polygon in the hyperbolic plane. Therefore the \emph{induced metric} on the boundary of $P$ is naturally a hyperbolic metric $m$ on the $N$-punctured sphere; it is a complete metric. Further, the labeling of the vertices, the equator, and the top and bottom of $P$ determine an identification (up to isotopy) of the surface of $P$ with the $N$-puncture sphere $\Sigma_{0,N}$, making $m$ into a point of the Teichm\"uller space $\Teich_{0,N}$. The marking also identifies the $1$--skeleton of $P$ with a graph $\Gamma$ on $\Sigma_{0,N}$ with vertices at the punctures.
The edges of the equator project to exterior edges of $p_L$ (resp. $p_R$) and top/bottom edges project to interior edges of $p_L$ (resp. $p_R$). 
We may assume the $1$--skeleton is a triangulation by adding additional top/bottom edges as needed. Consider an edge $e = z_1z_2$ adjacent to two faces $T = \Delta z_1 z_2 z_3$ and $T' = \Delta z_4 z_1 z_2$, each oriented so that the normal points out of $P$. Then the cross ratio $z = (z_1, z_2; z_3, z_4)$ contains the following information:

\begin{prop}
The edge $e$ is an equatorial edge if and only if $z = a + b \tau$ has real part $a > 0$.
\end{prop} 

\noindent Since $z$ is space-like, we may express it as  $$z = \pm e^{s + \tau \theta} := \pm e^s(\cosh \theta + \tau \sinh \theta).$$
By convexity of $P$, the imaginary part of $z$ is always positive. Hence, either $z = +e^{s + \tau \theta}$ with $\theta > 0$, or $z = -e^{s+\tau \theta}$ with $\theta < 0$. 
In the former case, the edge $e$ is a top/bottom edge and in the latter case, $e$ is an equatorial edge. 
In either case, $s = s(e)$ is precisely the \emph{shear coordinate} of the induced metric $m$ along the edge~$e$, and $\theta$~is the exterior \emph{dihedral angle} at the edge~$e$.

We now give the fundamentally important relationship between shearing and bending in the setting of ideal polyhedra. Let $m_L$ (resp. $m_R$) denote the double of $p_L$ (resp. $p_R$). Since the vertices of $P$, and its projections $p_L$ and $p_R$, are labeled, we may regard $m_L$ and $m_R$ as points of the Teichm\"uller space $\Teich_{0,N}$; we call $m_L$ the \emph{left metric} and $m_R$ the \emph{right metric}. Recall the definition of $\AdSPoly_N$ given in Section \ref{sec:intro-AdS-param}.

\begin{theorem}\label{thm:earthquakes}
Let $m_L, m_R, m \in \Teich_{0,N}$ be the left metric, the right metric, and the induced metric defined by $P \in \AdSPoly_N$, and let $\theta$ denote the dihedral angles. Then the following diagram holds: 
\begin{equation}\label{eqn:earthquake-diagram}
m_L \xrightarrow[]{E_\theta} m \xmapsto[]{E_\theta} m_R,
\end{equation}
 where $E_\theta$ denotes shearing along $\Gamma$ according to the weights $\theta$ (a positive weight means shear to the left). Further, given the left and right metrics $m_L$ and $m_R$ (any two metrics obtained by doubling two ideal polygons $p_L$ and $p_R$), the induced metric $m$ and the dihedral angles $\theta$ are the unique metric and weighted graph on $\Sigma_{0,N}$ (with positive weights on the top/bottom edges) such that~\eqref{eqn:earthquake-diagram} holds.
\end{theorem}

\begin{proof}
Let $\Gamma \in \Graph(\Sigma_{0,N}, \gamma)$ represent the $1$--skeleton of $P$.
By adding extra edges if necessary, we may assume $\Gamma$ is a triangulation.
As above we associate the shape parameter $z = \varepsilon e^{s(\alpha)+ \tau \theta(\alpha)}$ to a given edge~$\alpha$ of $\Gamma$, where $\varepsilon = \pm 1$.  Then,
\begin{align*}
z &= \varepsilon e^{s(\alpha)} (\cosh \theta(\alpha) + \tau \sinh \theta(\alpha))\\
&= \varepsilon e^{s(\alpha)} \left( \frac{1-\tau}{2} e^{-\theta(\alpha)} + \frac{1+\tau}{2} e^{\theta(\alpha)}\right)\\
&= \frac{1-\tau}{2}\varepsilon e^{s(\alpha) - \theta(\alpha)} +  \frac{1+\tau}{2}\varepsilon e^{s(\alpha)+\theta(\alpha)} 
\end{align*}
Therefore the shear coordinates in the left metric $m_L$ are given by $s_L = s- \theta$ 
and the shear coordinate in the right metric $m_R$ are $s_R = s + \theta$.
Equation~\eqref{eqn:earthquake-diagram} follows. The uniqueness statement also follows from this calculation. Indeed, given $m_L$, $m_R$ and any graph $\Gamma \in \Graph(\Sigma_{0,N})$ we may solve for the shear coordinates $s$, determining a metric $m$, and the weights $\theta$ needed to satisfy~\eqref{eqn:earthquake-diagram}. Specifically, $s  = (s_R + s_L)/2$ and $\theta = (s_R - s_L)/2$, where now $s_L$ and $s_R$ denote the shear coordinates with respect to $\Gamma$. We may construct a polyhedral embedding of $\Sigma_{0,N}$ whose induced metric is $m$ and whose (exterior) bending angles are $\theta$ as follows. Begin with the polyhedral embedding of $\Sigma_{0,N}$ into a space-like plane given by doubling $p_L$. Then bend this embedding along the edges of $\Gamma$ according to the weights $\theta$; note that this can be done consistently because $\theta$ satisfies condition (ii) in the definition of $\Angles$ (Section~\ref{sec:intro-AdS-param}) (because the shear coordinates for $m_L$ and $m$ satisfy that condition). If $\theta$ is positive on the top/bottom edges and negative on the equatorial edges, this polyhedral embedding is convex; it is the boundary of a convex ideal polyhedron $P$. By the definition of $s$ and $\theta$, we have that the left and right metrics of $P$ are precisely $m_L$ and $m_R$. The uniqueness statement follows because $P$ is uniquely determined by $m_L$ and $m_R$.
\end{proof}

As a corollary we obtain a version of Thurston's earthquake theorem for ideal polygons in the hyperbolic plane. 
A \emph{measured lamination} on the standard ideal $N$-gon is simply a pairwise disjoint collection of diagonals with \emph{positive} weights. We denote by $\ML_N$ the complex of these measured laminations. A function $\theta \in \Angles_{\Gamma}$ determines two measured laminations $\theta_+$ and~$\theta_-$ by restriction to the top edges of $\Gamma$ and to the bottom edges.

\begin{cor}[Earthquake theorem for ideal polygons]\label{cor:earthquake}
Let $p_L, p_R \in \poly_N$ be two ideal polygons. Then there exists unique $\theta_+, \theta_- \in \ML_N$ such that $p_R = E_{\theta_+} p_L$ and $p_L = E_{\theta_-} p_R$, where again $E_\lambda$ denotes shearing according to the edges of $\lambda \in \ML_N$ according to the weights of $\lambda$.
\end{cor}

\begin{proof}
Let $x_1, \ldots, x_N$ be the ideal vertices of $p_L$ and let $y_1, \ldots, y_N$ be the ideal vertices of $p_R$. Then, the vertices $z_i = \frac{1-\tau}{2}x_i + \frac{1+\tau}{2} y_i$ define an ideal polyhedron $P \in \AdSPoly_N$ such that $\varpi_L(P) = p_L$ and $\varpi_R(P) = p_R$.  We think of $\Sigma_{0,N}$ as the double of the standard ideal $N$-gon, meaning that the top hemisphere is identified with the standard ideal $N$-gon and the bottom hemisphere is identified with the standard ideal $N$-gon but with orientation reversed.  The left  metric $m_L$ (resp. $m_R)$ is obtained from $p_L$ (resp. $p_R$) by doubling. This means that the restriction of $m_L$ to the top hemisphere of $\Sigma_{0,N}$ is $p_L$ and the restriction of $m_L$ to the bottom hemisphere is $\overline{p_L}$, the same ideal polygon but with opposite orientation. Similarly, the restriction of $m_R$ to the top and bottom hemispheres of $\Sigma_{0,N}$ is $p_R$ and $\overline{p_R}$.
Let $\Gamma \in \Graph(\Sigma_{0,N}, \gamma)$ denote the $1$--skeleton of $P$ and let $\theta \in \RR^{E(\Gamma)}$ denote the dihedral angles.  
Theorem~\ref{thm:earthquakes} implies that $m_R = E_{2 \theta} m_L$.  Restricting to the top hemisphere, we have that $p_R = E_{\theta_+} p_L$ where $\theta_+ \in \ML_N$ is twice the restriction of $\theta$ to the top hemisphere. Restricting to the bottom hemisphere, we have that $\overline{p_R} = E_{\theta_-} \overline{p_L}$, where $\theta_-$ is the restriction of $\theta$ to the bottom hemisphere. This implies that $p_R = E_{-\theta_-} p_L$, or equivalently $p_L = E_{\theta_-} p_R$. Uniqueness of $\theta_+, \theta_-$ follows from uniqueness of $\theta$ in Theorem~\ref{thm:earthquakes}.
\end{proof}

\begin{Remark}
In the setting of closed surfaces, it is known~\cite{bon_fix} that given two filling measured laminations $\theta_+$ and $\theta_-$, there exists two hyperbolic surfaces $\rho_L$ and $\rho_R$ such that $\rho_R$ is obtained from $\rho_L$ by left earthquake along $\theta_+$ and also by right earthquake along $\theta_-$, and it is conjectured~\cite{mes_lor} that $\rho_L$ and $\rho_R$ are unique.
The analogous question, in the context of Corollary~\ref{cor:earthquake}, of whether a given $\theta_+$ and $\theta_-$ are realized by some $p_L$ and $p_R$, and whether they are realized uniquely, is an interesting one. A necessary condition is that $\theta_+$ and $\theta_-$ be filling, which means that any lamination intersects $\theta_+$ or $\theta_-$ transversely; this is equivalent to the statement that the graph $\Gamma \in \Graph(\Sigma_{0,N}, \gamma)$, obtained by placing the support of $\theta_+$ on the top hemisphere and the support of $\theta_-$ on the bottom hemisphere, is three-connected.
It will follow from Theorem~\ref{thm:main-AdS-angles} that in the case $N$ is odd, the polygons $p_L, p_R$ are unique, given the measured laminations $\theta_+, \theta_-$. This is because $\theta \in \Angles$ is determined entirely by its restrictions $\theta_+$ and $\theta_-$ to the top and bottom edges. However, there are examples of filling measured laminations $\theta_+, \theta_-$ such that there is no element $\theta \in \Angles$ whose restriction to the top edges is $\theta_+$ and whose restriction to the bottom edges is $\theta_-$ (see Appendix~\ref{sec:transitional-proof}).
The situation is even worse in the case $N$ is even. There is a one dimensional family of pairs $p_L, p_R$ for which the laminations $\theta_+, \theta_-$ turn out to be the same. This is because for any $\theta \in \Angles$, there is a one parameter family of deformations of $\theta$ which leave $\theta_+, \theta_-$ unchanged: simply add and subtract the same quantity from the weights of alternating edges on the equator. Further, in the case $N$ even, only a codimension one subspace of filling laminations $\theta_+, \theta_-$ are realized in Corollary~\ref{cor:earthquake}. It is an interesting problem to determine this codimension one subspace.
\end{Remark}

\subsection{The pseudo-complex structure on $\AdSPoly_N$}\label{sec:pseudo}
The space of marked ideal polyhedra $\AdSPoly_N$ naturally identifies with a subset of $(\RR + \RR \tau)^{N-3}$, by transforming each ideal polyhedron so that its first three vertices are respectively $0, 1, \infty \in \mathbb P^1 \BB$. The marking on each polyhedron $P \in \AdSPoly_N$ identifies $P$ with the standard $N$-punctured sphere $\Sigma_{0,N}$. So, given a triangulation $\Gamma$ on $\Sigma_{0,N}$ with vertices at the punctures and edge set denoted $E$, we may define the map $z_\Gamma: \AdSPoly_N \to (\RR + \RR \tau)^E$ which associates to each edge $e$ of a polyhedron $P$ the cross ratio of the four points defining the two triangles adjacent at $e$. This map is pseudo-complex holomorphic, meaning that the differential is $(\RR+ \RR\tau)$--linear. This observation allows us to prove the following analogue of Theorem~\ref{thm:H3-duality}.

\begin{theorem}\label{thm:duality}
A polyhedron $P \in \AdSPoly_N$ is infinitesimally rigid with respect to the induced metric if and only if $P$ is infinitesimally rigid with respect to the dihedral angles.
\end{theorem}
\begin{proof}
Let $V \in T_P \AdSPoly_N \cong (\Rtau)^{N-3}$. Let $\Gamma$ be a triangulation obtained from the $1$--skeleton of $P$ by adding edges in the non-triangular faces if necessary. Since the induced metric is determined entirely by the shear coordinates with respect to $\Gamma$, we have that $V$ does not change the induced metric to first order if and only if  $d\log z_\Gamma (V)$ is pure imaginary. On the other hand, $V$ does not change the dihedral angles to first order if and only if $d\log z_\Gamma (V)$ is real. Therefore $V$ does not change the induced metric if and only if $\tau V$ does not change the dihedral angles.
\end{proof}

\subsection{Half-pipe geometry in dimension three}\label{s:HP3-geometry}

We give some lemmas useful for working with $\HP^3$.
Recall the algebra $\RR + \RR\sigma$, with $\sigma^2 = 0$. The half-pipe space is given by 
$$\HP^3:= \mathbb X = \left\{ X + Y \sigma : X,Y \in M_2(\RR), X^T = X, \det(X) > 0, Y^T = -Y\right\}/\sim,$$
where $(X + Y \sigma) \sim \lambda (X + Y \sigma)$ for $\lambda \in \RR^\times$. 
There is a projection $\varpi : \HP^3 \to \HH^2$, defined by $\varpi(X + Y \sigma) = X$, where we interpret the symmetric matrices $X$ of positive determinant, considered up to scale, as a copy of $\HH^2$. The fibers of this projection will be referred to simply as \emph{fibers}. The projection can be made into a diffeomorphism $\mathbb X \to \HH^2 \times \RR$ (not an isometry) given in coordinates by 
\begin{align}\label{product-coords-dim3}
X + Y \sigma \mapsto (X, L),
\end{align} where the length $L$ along the fiber is defined by the equation 
\begin{align} \label{fiber-length-dim3}
Y &= L \sqrt{\det X}\begin{pmatrix} 0 & -1\\ 1 & 0 \end{pmatrix}. 
\end{align}
The ideal boundary $\partial_\infty \XX_0$ identifies with $\mathbb P^1 (\Rsigma)$, which identifies with the tangent bundle $T \RP^1$ via the natural map $T \RR^2 \to (\RR + \RR\sigma)^2$ sending a vector $v \in \RR^2$ and a tangent vector $w \in T_v \RR^2 = \RR^2$ to $v + \sigma w$. It will be convenient to think of an ideal vertex as an infinitesimal variation of a point on $\RP^1 \cong \partial_\infty \HH^2$. In this way, a convex ideal polyhedron $P$ in $\HP^3$ defines an infinitesimal deformation $V = V(P)$ of the ideal polygon $p = \varpi(P)$ in $\HH^2$.

We restrict to the identity component of the structure group, which is given by
\begin{align*}
\mathbb G_0 &= \PSL(2, \RR + \RR\sigma)\\
&= \{ A + B \sigma : A \in \SL(2,\RR), \text{ and } B \in T_A \SL(2,\RR) \} / \pm.
\end{align*}
The structure group identifies with the tangent bundle $T \PSL(2,\RR)$, and it will be convenient to think of its elements as having a finite component $A\in \PSL(2,\RR)$ and an infinitesimal component $a \in \mathfrak{sl}(2,\RR)$, via the isomorphism
\begin{align*}
\PSL(2,\RR) \ltimes \mathfrak{sl}(2,\RR) &\to \mathbb G_0\\
(A,a) &\mapsto A + Aa \sigma,
\end{align*}
where $Aa \in T_A \PSL(2,\RR)$.
(This is the usual isomorphism $G \ltimes \mathfrak g \to T G$ for a Lie group $G$ with Lie algebra $\mathfrak g = T_1 G$, gotten by left translating vectors from the identity.)
The identification $\mathbb G_0 \cong T \PSL(2,\RR)$ is compatible with the identification $\partial_\infty \HP^3 \cong T \RP^1$.

Thinking of $a \in \mathfrak{sl}(2,\RR)$ as an infinitesimal isometry of $\HH^2$, recall that at each point $X \in \HH^2$ we may decompose $a$ into its translational ($X$-symmetric) and rotational ($X$-skew) parts: 
\begin{align*}
a &= a_{X\text{-sym}} + a_{X\text{-skew}} \\ 
&:= \frac{1}{2}\left( a + Xa^T X^{-1} \right) + \frac{1}{2} \left( a - Xa^T X^{-1} \right),
\end{align*}
where the rotational part $a_{X\text{-skew}}$ is a rotation centered at $X$ of infinitesimal angle $\rot(a,X)$ defined by
$$\sqrt{X}^{-1} a_{X\text{-skew}} \sqrt{X} = \rot(a,X) \begin{pmatrix} 0 & -1/2 \\ 1/2 & 0\end{pmatrix}.$$
 The action of an element of $\mathbb G_0$ in the fiber direction depends on the rotational part of the infinitesimal part of that element.
\begin{figure}[h]
{\centering

\def\svgwidth{3.7in}
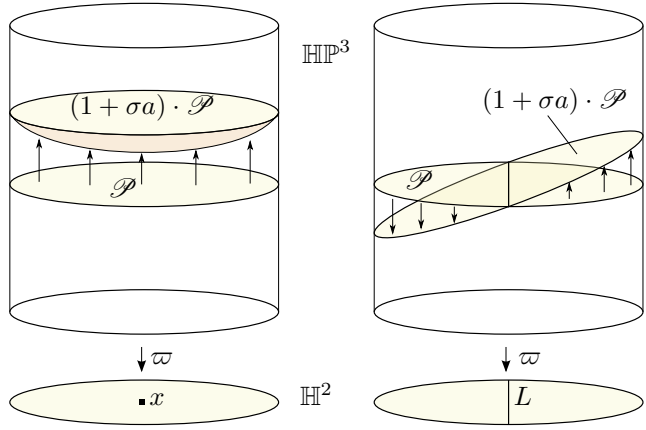

\caption{The action of $1+a\sigma$ on $\HP^3$ when $a$ is an infinitesimal rotation centered at $x$ (left), or $a$ is an infinitesimal translation along $L$ (right).\label{fig:HP-action}}
}
\end{figure}

\begin{Lemma}
\label{lem:HP3-action}
The action of a pure infinitesimal $1 + a\sigma$ on the point $X + Y\sigma \in \mathbb X$ is by translation in the fiber direction by amount equal to the rotational part $\rot(a, X)$ of the infinitesimal isometry $a$ at the point $X \in \HH^2$.  In the product coordinates~(\ref{product-coords-dim3}):
$$1+ a \sigma : (X,L) \mapsto (X, L + \rot(a,X)).$$
More generally, the action of $A + Aa\sigma$ is given by 
$$A + Aa\sigma : (X,L)\mapsto (A\cdot X, L + \rot(a,X)).$$
\end{Lemma}

\begin{proof}
\begin{align*}
(1+a\sigma)\cdot (X + Y\sigma) &= (1+a\sigma) (X+  \sigma Y) (1 - a^T\sigma)\\
&= X + \sigma Y + \sigma(a X - X a^T)\\
&= X + \sigma Y + \sigma \ 2 a_{X\text{-skew}} X\\
&= X + \sigma Y + \sigma \ 2 \rot(a,X) \sqrt{X} \begin{pmatrix} 0 & -1/2\\ 1/2 & 0\end{pmatrix} \sqrt{X}\\
&= X + \sigma Y + \sigma \  \rot(a,X) \det(\sqrt{X}) \begin{pmatrix} 0 & -1\\ 1 & 0\end{pmatrix},
\end{align*}
and the first statement now follows from Equation~(\ref{fiber-length-dim3}). The second more general formula follows easily after left multiplication by $A$.
\end{proof}

\begin{defi}
Let $a \in \mathfrak{sl}_2 \RR$ be an infinitesimal translation of length $t$ along an oriented geodesic $\ell$ in $\HH^2$. Then, for any oriented geodesic $\wt{\ell}$ in $\HP^3$ that projects to $\ell$, the element $1+ a \sigma$ is called an \emph{infinitesimal rotation} about the axis $\wt{\ell}$ of infinitesimal angle $t$.
\end{defi}

Thinking of the fiber direction in $\HP^3$ as the direction of infinitesimal unit length normal to $\HH^2$ into either $\HH^3$ or $\AdS^3$, the definition is justified by the previous lemma. In fact, the amount of translation in the fiber direction is $t$ times the signed distance to $\widetilde \ell$.  

\subsection{Ideal polyhedra in $\HP^3$}\label{subsec:polyhedra-HP}

There are several important interpretations of a convex ideal polyhedron $P$ in $\HP^3$. As described in the previous section, $P$ defines an infinitesimal deformation $V = V(P)$ of the ideal polygon $p = \varpi(P)$ in $\HH^2$. Alternatively, $P$ may be interpreted as an infinitesimally thick polyhedron in $\HH^3$ or $\AdS^3$.
 Multiplying the tangent vector $V$ by $i$ (resp. $\tau$) describes an infinitesimal deformation $iV$ (resp. $\tau V$) of the polygon $p$ into $\HH^3$ (resp. $\AdS^3$). The polyhedron $P$ in $\HP^3$ is a rescaled limit of a path of hyperbolic (resp. anti-de Sitter) polyhedra collapsing to $p$ and tangent to $iV$ (resp. $\tau V$) in the following sense. Consider the path of algebras $\BB_t$ generated by $\kappa_t$ such that $\kappa_t^2 = -t|t|$. Then the geometries $\XX(\BB_t)$ associated to these algebras are conjugate to $\XX(\BB_1) = \XX(\CC) = \HH^3$ for all $t > 0$, or to $\XX(\BB_{-1}) = \XX(\Rtau) = \AdS^3$ for $t < 0$. For $t > 0$, the map $\mathfrak a_t : \CC \to \BB_t$ defined by $i \mapsto \kappa_t/|t|$ is an isomorphism of algebras. For $t < 0$, the map $\mathfrak a_t : \Rtau \to \BB_t$ defined by $\tau \mapsto \kappa_t/|t|$ is an isomorphism. Each of these maps defines a projective transformation, again denoted $\mathfrak a_t$, taking the standard model of hyperbolic space $\HH^3 = \XX(\BB_1)$ (resp. the standard model of anti-de Sitter space $\AdS^3 = \XX(\BB_{-1})$) to the conjugate model $\XX(\BB_t)$.
 
\begin{prop}\label{prop:limit}
Consider a smooth family $Q_t$ of ideal polyhedra in $\HH^3$ (resp. $\AdS^3$), defined for $t > 0$ (resp. for $t < 0$). Assume that $Q_0 = p$ is an ideal polygon contained in the central hyperbolic plane $\plane$ bounded by $\RP^1$ and $Q'_0 = U + iW$ (resp. $Q_0' = U + \tau W$), where $U, W$ are infinitesimal deformations of $p$ as an ideal polygon in $\HH^2$. Then the limit of $\mathfrak a_t (Q_t)$ as $t \to 0$ is an ideal polyhedron $P$ in $\XX(\BB_0) = \HP^3$ which satisfies $\varpi(P) = Q_0$ and $V(P) = W$.
\end{prop}

The interplay between these two interpretations leads to Theorem~\ref{thm:HP-bend} below, which is a fundamental tool for studying half-pipe geometry. Before stating the theorem, let us recall the terminology introduced in Section~\ref{sec:HP-param} and state a proposition.
We fix an orientation of the fiber direction once and for all. 
Every convex ideal polyhedron in $\HP^3$ has a top, for which the outward pointing fiber direction is positive, and a bottom, for which the outward pointing fiber direction is negative. The edges naturally sort into three types: an edge is called a \emph{top edge} if it is adjacent to two top faces or a \emph{bottom edge} if it is adjacent to two bottom faces, or an \emph{equatorial edge} if it is adjacent to both a top and bottom face. The union of the top faces is a bent polygon which projects down to the ideal polygon $p = \varpi(P)$ in~$\HH^2$. The union of the bottom faces also projects to $p$. The \emph{infinitesimal dihedral angle} at an edge is measured in terms of the infinitesimal rotation angle needed to rotate one face adjacent to the edge into the same plane as the other. The dihedral angle at a top/bottom edge will be given a positive sign, while the dihedral angles at an equatorial edge will be given a negative sign. This sign convention is justified by the following (see \cite[\S 4.2]{dan_age}):

\begin{prop}\label{prop:angles-limit}
The infinitesimal dihedral angle along an edge of $P$ is simply the derivative of the dihedral angle of the corresponding edge of $Q_t$, where $Q_t$ is as in Proposition~\ref{prop:limit}.
\end{prop}

Alternatively, dihedral angles may also be measured using the cross ratio. Indeed, if two (consistently oriented) ideal triangles $T = \Delta z_1 z_2 z_3$ and $T' = \Delta z_4 z_1 z_2$ meet at a common edge $\alpha = z_1 z_2$, then the cross ratio $z = (z_1, z_2; z_3, z_4)$ satisfies that $z = \varepsilon e^{s+ \sigma \theta} = \varepsilon e^s(1+ \sigma \theta)$, where $s$ is the shear between $T$ and $T'$, where $\theta$ is the dihedral angle, and where $\varepsilon$ is $+1$ if $\alpha$ is an edge of the equator and $-1$ if $\alpha$ is a top/bottom edge. 

We consider the bending angles on the top (resp. bottom) edges of an ideal polyhedron $P$ as a (positive) measured lamination on the ideal polygon $p = \varpi(P)$. 
The following theorem is the infinitesimal version of Theorem~\ref{thm:earthquakes} about the interplay between earthquakes and AdS geometry.

\begin{theorem}\label{thm:HP-bend}
Let $P$ be an ideal polyhedron in $\HP^3$ and let $\theta_{+}$ (resp. $\theta_{-}$) be the measured lamination on $p = \varpi(P)$ describing the bending angles on top (resp. on bottom). Then the infinitesimal deformation $V = V(P)$ of $p$ defined by $P$ is equal to $e_{\theta_{+}}(p)$, where $e_{\theta_{+}}$ is the infinitesimal left earthquake along $\theta_{+}$. Similarly, $V = -e_{\theta_{-}}(p)$ is obtained by right earthquake along $\theta_{-}$.
\end{theorem}

\begin{proof}
Let $\Gamma \in \Graph(\Sigma_{0,N}, \gamma)$ represent the $1$--skeleton of $\partial P$. By adding extra edges if necessary, we may assume $\Gamma$ is a triangulation.
As above we associate the shape parameter $z(\alpha)  = \varepsilon e^{s(\alpha)+ \sigma \theta(\alpha)}$ to any given edge $\alpha$ of $\Gamma$. Note that the map taking four points on $\RP^1$ to their cross ratio is smooth and that the isomorphism $T \RP^1 \cong \mathbb P^1(\RR + \RR \sigma)$ commutes with the cross ratio operation. Therefore the shear coordinate of $p = \varpi(P)$ at $\alpha$ is $s(\alpha)$ and the infinitesimal variation of the shear coordinate at $\alpha$ under the deformation $V(P)$ is $\theta(\alpha)$. The result follows.
\end{proof}

\subsection{Half-pipe geometry in dimension two}
The structure group $\mathbb G$ for $\HP^3$ acts transitively on degenerate planes, i.e. the planes for which the restriction of the metric on $\HP^3$ is degenerate. These are exactly the planes that appear vertical in the standard picture of $\HP^3$ (as in Figure~\ref{fig:HP-action}); they are the inverse image of lines (copies of $\HH^1$) in $\HH^2$ under the projection~$\varpi$. Each degenerate plane is a copy of two-dimensional half-pipe geometry $\HP^2$. For the purposes of the following discussion, we will fix one degenerate plane in $\HP^3$ as our model:
$$ \HP^2 := \left\{ \begin{pmatrix} x & 0 \\ 0 & x^{-1} \end{pmatrix} + \sigma\begin{pmatrix} 0 & y \\ -y & 0 \end{pmatrix} \right\}.$$
Here we describe two important facts about $\HP^2$. The first is (reasonably) named the \emph{infinitesimal Gauss-Bonnet formula}. See \cite[\S 3]{dan_age} for details about half-pipe geometry in arbitrary dimensions.

There is an invariant notion of area in $\HP^2$. As above, let $L$ denote the length function along the fiber direction. Then the area of a polygon $p$ (or a more complicated body) is the integral of the length $L(\varpi^{-1}(x) \cap p)$ of the segment of $p$ above $x$, over all $x \in \varpi(p) \subset \HH^1$. Alternatively, if $p$ is the limit as $t \to 0$ of $\mathfrak a_t p_t$, where $p_t$ is a smooth family of collapsing polygons in $\HH^2$, then the area of $p$ is simply derivative at $t=0$ of the area of $p_t$.

\begin{prop}[Infinitesimal Gauss-Bonnet formula]\label{prop:inf-GB}
Let $p$ be a polygon in $\HP^2$ whose edges are each non-degenerate. Then the area of $p$ is equal to the sum of the exterior angles of $p$. In particular, the sum of the exterior angles of any polygon is positive.
\end{prop}

\begin{proof}
Let $p_t$ be a smooth family of collapsing polygons in $\HH^2$ so that $p$ is the limit as $t \to 0$ of $\mathfrak a_t p_t$. Then the area of $p$ is the derivative of the area of $p_t$ at $t=0$. Each exterior angle of $p$ is the derivative of the corresponding angle of $p_t$ at $t= 0$. The proposition follows from the usual Gauss-Bonnet formula for polygons in $\HH^2$.
\end{proof}

Secondly, we give a bound on the dihedral angle between two non-degenerate planes in terms of the angle seen in the intersection with a degenerate plane $H \cong \HP^2$. This will be used in the proof of Proposition~\ref{prop:imageinA-HP}.

\begin{prop}\label{prop:HP-angle-bound}
Let $P, Q$ be two non-degenerate planes in $\HP^3$ which intersect at dihedral angle $\theta$. Let $H$ be a degenerate plane so that the lines $H \cap P$ and $H \cap Q$ intersect at angle $\vartheta$ in $H \cong \HP^2$.  Then $\operatorname{sign}(\vartheta) = \operatorname{sign}(\theta)$ and $|\vartheta| \leq |\theta|$ with equality if and only if $H$ is orthogonal to the line $P \cap Q$.
\end{prop}

\begin{proof}
We may change coordinates so that $P = \plane$ (recall that $\plane$ is a copy of $\HH^2$ common to all of the models $\XX(\BB)$ in projective space, see Remark~\ref{rem:plane}). The second plane $Q$ is the limit as $t \to 0$ of $\mathfrak a_t Q_t$, where $Q_t$ is a smoothly varying family of planes in $\HH^3$ with limit $Q_0 = \plane$. We may choose the path $Q_t$ so that the line $L = Q_t \cap \plane$ is constant for all $t > 0$. The dihedral angle between $Q$ and $\plane$ is the derivative at $t= 0$ of the dihedral angle $\theta_t$ between $Q_t$ and $\plane$, now thought of as a plane in $\HH^3$. The degenerate plane $H$ defines a plane $H'$ in (the projective model of) $\HH^3$ which is orthogonal to $\plane$. Let $\vartheta_t$ be the angle formed by $Q_t \cap H'$ and $\plane \cap H'$ in $H' \cong \HH^2$. Then, because $H'$ and $\plane$ are orthogonal, we have that $$\tan \vartheta_t = \tan \theta_t \sin \varphi$$
where $0 < \varphi \leq \frac{\pi}{2}$ is the angle between the line $L = Q_t \cap \plane$ and $H'$.
The proposition now follows since $\vartheta = \frac{d}{dt} \big|_{t=0} \vartheta_t$, $\theta = \frac{d}{dt} \big|_{t=0} \theta_t$ and $\theta_0 = \vartheta_0$ (both are either zero or~$\pi$).
\end{proof}

%
%
%

\section{Length functions and earthquakes}\label{sec:lengthfunctions}

We prove Theorem~\ref{thm:main-HP} by showing that each ideal polyhedron in $\HP^3$ is realized as the unique minimum of a certain length function defined in terms of its dihedral angles.
Our strategy is inspired by a similar one used by Series~\cite{ser_ker}, and later Bonahon~\cite{bon_kle}, in the setting of quasifuchsian hyperbolic three-manifolds with small bending.

\subsection{Shear and length coordinates on the Teichm\"uller space of a punctured sphere}

Consider an ideal triangulation $\Gamma$ of the $N$-times punctured sphere $\Sigma_{0,N}$. Let $\alpha_1, \ldots, \alpha_n$ denote the $n=3N-6$ edges of $\Gamma$. There are two natural
coordinate systems on the Teichm\"uller space $\Teich_{0,N}$ of complete hyperbolic metrics on $\Sigma_{0,N}$ (see \cite{penner:decorated,thurston:minimal}):
\begin{itemize}
\item Let $s_1,\cdots, s_n$ denote the shear coordinates
along the edges of $\Gamma$. 
The sum of the shear coordinates over edges adjacent to a particular vertex is always zero. Under this
condition, the shears along the edges provide global coordinates on $\Teich_{0,N}$.
\item We may define length coordinates $\ell_1, \ldots, \ell_n$ on $\Teich_{0,N}$ as follows.
In any hyperbolic structure, choose a horocycle around each cusp, and let $\ell_i$ denote the (signed) length of the segment of $\alpha_i$ connecting the two relevant horocycles. By abuse, we call $\ell_i$ the length of $\alpha_i$.
Changing a horocycle at a particular cusp corresponds 
to adding a constant to the lengths of all edges going into that cusp. The lengths $\ell_1, \ldots, \ell_n$ are only
well-defined up to this addition of constants, making these coordinates elements of $\mathbb R^n/\mathbb R^N$.
\end{itemize}
\noindent It is well-known \cite{penner:decorated,thurston:minimal}
that both the shears and the lengths give global coordinate systems for Teichm\"uller space. 
It is quite simple to go from length coordinates to shear coordinates, in fact the map sending lengths to shears is linear.
To describe this coordinate transformation more precisely, 
let us establish some notation. The orientation of the surface determines a cyclic order on the edges of any triangle.
Given any two edges $\alpha_i, \alpha_j$, let $\epsilon_{ij} = -\epsilon_{ji}$
 be the number of positively oriented triangles $T$ of $\Gamma$ such that $\alpha_i, \alpha_j$ are distinct edges of $T$ counted with a positive sign if $\alpha_j$ follows $\alpha_i$ in the cyclic order on the edges of $T$, and with negative sign if $\alpha_i$ follows $\alpha_j$. By definition, $(\epsilon_{ij})_{1\leq i,j\leq n}$ is an
anti-self adjoint matrix with entries in $\{-1,0,1\}$.
It is straightforward to check the following:
\begin{lemma}[Thurston~{\cite[p.~44]{SOP}}] \label{lm:penner}
Given a hyperbolic metric $h\in \Teich_{0,N}$ with length coordinates $(\ell_i)$,
the corresponding shear coordinates are defined by
$$s_i =\frac{1}{2} \sum_j \epsilon_{ij}\ell_j~. $$
Note that the right-hand side is independent of the horocycles chosen to define the~$\ell_i$.
\end{lemma}

\begin{defi}
Let $\omega$ denote the anti-symmetric bilinear form on $\Teich_{0,N}$, defined by
\begin{equation}\label{eqn:omega1}
\omega = \frac{1}{2}\sum_{i,j} \epsilon_{ij} d \ell_i \otimes d \ell_j.
\end{equation}
Note that, by Lemma~\ref{lm:penner}, we may also express $\omega$ as
\begin{equation} \label{eqn:omega}
\omega = \sum_{i} d \ell_i \otimes d s_i.
\end{equation}
It follows that $\omega$ is well-defined (independent of the ambiguity in the definition of $d\ell_i$) because for any tangent vector $Y$, $d s_i(Y)$ is a \emph{balanced function} on the set $E = E(\Gamma)$ of edges, meaning it is a function whose values sum to zero on those edges incident to any vertex.
\end{defi}
\noindent From the second expression for $\omega$, we can see that it is a symplectic form, i.e. it is non-degenerate. In fact, we mention that $\omega$ is nothing other than (a multiple of) the Weil-Petersson symplectic form (see Wolpert~\cite{wolpert:products} and Fock-Goncharov \cite{fock-goncharov-1}), though we will not need this fact.
It is straight-forward to check directly that $\omega$ does not depend on 
the particular triangulation used in its definition.

\subsection{The gradient of the length function}

Given a function $f:\Teich_{0,N}\to \R$, we denote by $D^\omega f$ its symplectic gradient
with respect to $\omega$, defined by the following relation: for any vector field
$X$ on $\Teich_{0,N}$, 
$$ \omega(D^\omega f,X) = df(X).$$

Let $\theta = (\theta_1, \ldots, \theta_n)$ be any balanced assignments of weights to the edges of $\Gamma$. Then one may define the corresponding \emph{length function} $\ell_\theta$ as a function on
$\Teich_{0,N}$: for any hyperbolic metric $h\in \Teich_{0,N}$, with length coordinates $(\ell_i)_{1\leq i\leq n}$,
set
$$ \ell_\theta(h) = \sum_i \theta_i \ell_i~.$$
The function $\ell_\theta$ does not depend on the choice of horocycles at the cusps precisely because~$\theta$
is balanced. We let $e_\theta$ denote the vector field on $\Teich_{0,N}$ 
defined by $ds_i(e_\theta) = \theta_i$, in other words $e_\theta$ shears along each edge 
according to the weights~$\theta$. It follows immediately from \eqref{eqn:omega} that:
 
\begin{lemma} \label{lm:gradient}
Let $\theta = (\theta_1, \ldots, \theta_n)$ be balanced weights on the edges of $\Gamma$. Then
$$ D^\omega \ell_\theta = -e_\theta.$$
\end{lemma}

\subsection{The space of doubles is Lagrangian}
We assume, from here on, that our graph $\Gamma$ admits a Hamiltonian cycle $\gamma$. Then cutting $\Sigma_{0,N}$ along $\gamma$ yields two topological ideal polygons, one of which we label \emph{top} and the other \emph{bottom}. There is an orientation reversing involution $\iota$ on $\Sigma_{0,N}$ which exchanges top with bottom and point-wise fixes $\gamma$. We let $\doubles$ denote the half-dimensional subspace of $\Teich_{0,N}$ which is fixed by the action of $\iota$, i.e. those hyperbolic metrics which are obtained by doubling a hyperbolic ideal polygon and marking the surface in such a way that the boundary of the polygon identifies with~$\gamma$. 

\begin{prop}\label{pr:doubles}
The space of doubles $\doubles$ is a Lagrangian subspace of $\Teich_{0,N}$ with respect to $\omega$.
\end{prop}

\begin{proof}
We may compute $\omega$ with respect to a symmetric triangulation $\Gamma$ (one which is fixed under the involution $\iota$). For $h \in \Teich_{0,N}$, the shear coordinates $(s_i(h))$ are anti-symmetric, in the sense that, if $\iota(\alpha_i) = \alpha_j$, then $s_i(h) = -s_j(h)$. (So, in particular, $s_i(h) = 0$, if $\alpha_i$ is an edge of $\gamma$.) On the other hand, the lengths $(\ell_i(h))$ are symmetric, in the sense that, if $\iota(\alpha_i) = \alpha_j$, then $\ell_i(h) = \ell_j(h)$. The proposition follows immediately from the second expression \eqref{eqn:omega} for $\omega$ above.
\end{proof}

\subsection{Convexity of the length function}

We now show a form of convexity for the restriction of the length function $\ell_\theta$ to the space of doubles $\doubles$ in $\Teich_{0,N}$.
It will sometimes be convenient to identify the space of doubles $\doubles$ with the space $\poly = \poly_N$ of marked ideal polygons in the hyperbolic plane, and to think of (the restriction of) $\ell_\theta$ as a function on $\poly$.
The graph $\Gamma$ on $\Sigma_{0,N}$, then, projects to each polygon $p$ in $\poly$, with $\gamma$ identified to the perimeter edges of $p$ and all other edges of $\Gamma$ identified with diagonals of~$p$.

The following proposition is the analog, in the (simpler) setting of ideal polygons, of a theorem of Kerckhoff~\cite{ker_lin} which played a key role in Series's analysis of quasi-Fuchsian manifolds with small bending~\cite{ser_ker}. In a similar way, the proposition is crucial for Theorem~\ref{thm:main-HP}.

\begin{prop} \label{pr:convex}
For all $\theta\in \Angles_\Gamma$, the length function $\ell_\theta:\poly_N\to \R$ is proper and
admits a unique critical point which is a non-degenerate minimum.
\end{prop}

\noindent The proof is based on two lemmas.

\begin{lemma} \label{lm:proper}
If $\theta\in \Angles_\Gamma$, then $\ell_\theta:\poly\to \R$ is proper.
\end{lemma}

\begin{lemma} \label{lm:path-convex}
The function $\ell_\theta$ is convex and non-degenerate on earthquake paths in $\poly$.
\end{lemma}

\begin{proof}[Proof of Proposition \ref{pr:convex}]
Let $\theta\in \Angles_\Gamma$. Since $\ell_\theta$ is proper by Lemma \ref{lm:proper}, it has
at least one minimum in $\poly$. Moreover Lemma \ref{lm:path-convex} shows that any
critical point is a non-degenerate minimum.

Let $p,p'\in \poly_N$ be two minima of $\ell_\theta$. There is, by Corollary \ref{cor:earthquake},
a unique measured lamination $\lambda$ on $p$ such that $E_\lambda(p)=p'$. 
Then Lemma \ref{lm:path-convex} shows that the function 
$t\mapsto \ell_\theta(E_{t\lambda}(p))$ is convex and non-degenerate, so it cannot
have critical points both at $t=0$ and at $t=1$, a contradiction. So $\ell_\theta$
has a unique critical point on $\poly_N$.
\end{proof}

\noindent We now turn to the proofs of the two lemmas.

\begin{proof}[Proof of Lemma \ref{lm:proper}]
Let $(p_n)_{n\in \N}$ be a sequence of ideal polygons with $N$ vertices,
which degenerates in $\poly_N$. Then, after taking a subsequence, if necessary,
there is a finite collection of segments $a_1, \ldots, a_p$ on the polygon such that:
\begin{itemize}
\item $a_i$ and $a_j$ are disjoint, if $i\neq j$,
\item for all $n$, each $a_i$ is realized as a minimizing geodesic segment connecting 
two non-adjacent edges of $p_n$,
\item for all $i\in \{ 1,\cdots, p\}$, the length of $a_i$ in $p_n$ goes to zero, as $n \to\infty$,
\item any two edges of $p_n$ that can be connected by a segment disjoint from the 
$a_i$ remain at distance at least $\epsilon$, for some $\epsilon>0$ independent of $n$. 
\end{itemize}

After taking a further subsequence, the $p_n$ converge to the union of $p+1$ ideal polygons $p_\infty^{(1)}, \ldots, p_\infty^{(p+1)}$, which, topologically, is obtained by cutting the original polygon along each $a_i$ and then collapsing each (copy of each) segment $a_i$ to a new ideal vertex. 
Recall that given $r > 0$ and a geodesic line $\alpha$ in $\HH^2$, the $r$-neighborhood of $\alpha$ is called a \emph{hypercycle} neighborhood of $\alpha$.
We may choose horoballs at each ideal vertex and disjoint hypercycle neighborhoods $N_{i,n}$ of the (geodesic realization in $p_n$ of) $a_i$, with radii $r_{i,n} \to \infty$, which converge to a system of horoballs for the limiting ideal polygons $p_\infty^{(1)}, \ldots, p_\infty^{(p+1)}$. Our function $\theta$ is naturally defined on the limiting polygons, since all edges of the limit correspond to edges of the original polygon. However, $\theta$ is no longer balanced at the new ideal vertices of $p_\infty^{(1)}, \ldots, p_\infty^{(p+1)}$; instead the sum of the $\theta$ values along the edges going into one of the new vertices is strictly positive, since $\theta$ satisfies assumption $(iii)$ of Section~\ref{sec:intro-AdS-param}. Now, we may split $\ell_\theta$ into two pieces $$\ell_\theta= \ell_\theta \big |_{\cup N_{i,n}} + \ell_\theta \big |_{(\cup N_{i,n})^c},$$ corresponding to the weighted length contained in the union of the neighborhoods $N_{i,n}$ and the weighted length outside of those neighborhoods. The former is always positive, since $\theta$ satisfies condition $(iii)$ of Section~\ref{sec:intro-AdS-param}, and since the arcs with positive weight crossing $a_i$ have length at least $2 r_{i,n}$ in $N_{i,n}$, while the two arcs with negative weight crossing $a_i$ have length exactly $2 r_{i,n}$ in $N_{i,n}$. The later converges to the $\theta$-length function $\ell_\theta(p_\infty^{(1)})+ \cdots + \ell_\theta(p_\infty^{(p+1)})$ of the limiting polygons with respect to the limiting horoballs. However, by altering the radii of the neighborhoods $N_{i,n}$, we may arrange for the limiting horoball around each of the new vertices to be arbitrarily small (i.e. far out toward infinity), making $\ell_\theta(p_\infty^{(1)})+ \cdots + \ell_\theta(p_\infty^{(p+1)})$ arbitrarily large. It follows that $\ell_\theta(p_n) \to +\infty$.
\end{proof}

\begin{proof}[Proof of Lemma \ref{lm:path-convex}]
Let $p\in \poly_N$, and let $\lambda$ be a measured lamination on $p$, that is,
a set of disjoint diagonals $\beta_1, \cdots, \beta_q$ each with a weight $\lambda_i>0$.
We need to prove that the function $t\to \ell_\theta(E_{t\lambda}p)$ is 
convex with strictly positive second derivative. To prove this, we prove an analogue of the Kerckhoff--Wolpert formula in this setting, specifically:
\begin{equation}\label{eqn:KW}
\frac d{dt} \ell_\theta(E_{t \lambda}p) =  \sum \theta_i \lambda_j \cos(\varphi_{ij}) + K,
\end{equation}
where $\varphi_{ij} \in (0,\pi)$ is the angle at which the edge $\alpha_i$ of $\Gamma$ crosses the edge $\beta_j$ of the support of $\lambda$, the sum is taken over all $i,j$ so that $\alpha_i$ intersects $\beta_j$ non-trivially, and $K := K(\theta, \lambda)$ is independent of $p$ and $t$.
The lemma follows from this formula by a standard argument about earthquakes (see \cite[Lemma 3.6]{ker_nie}): each angle $\varphi_{ij}$  of intersection strictly decreases with $t$ because, from the point of view of the edge $\beta_j$, the endpoints at infinity of $\alpha_i$ are moving to the left.

It suffices to prove the formula~\eqref{eqn:KW} in the case that the lamination $\lambda$ is a single diagonal $\beta$ with weight equal to one. We choose horocycles $h_{v,t}$ at each vertex $v$ and at each time $t$ along the earthquake path as follows. Begin at time $t = 0$ with any collection of horocycles $\{h_{v,0}\}$. For a vertex $v$ that is not an endpoint of $\beta$, we simply apply the earthquake $E_{t \lambda}$ to $h_{v,0}$: Define $h_{v,t} = E_{t\lambda} h_{v,0}$. If $w$ is an endpoint of $\beta$, then the earthquake breaks the horocycle $h_{w,0}$ into two pieces. We define $h_{w,t}$ to be the horocycle equidistant from these two pieces. An easy calculation in hyperbolic plane geometry shows that, for $\alpha_i$ an edge of $\Gamma$ crossing $\beta$, we have 
$$ \frac d{dt} \ell(\alpha_i) = \frac{d}{dt} \operatorname{dist}(h_{v,t}, h_{v',t}) =  \cos(\varphi_{i}),$$
where $\varphi_{i}$ is the angle at which $\alpha_i$ crosses $\beta$, where $v$ and $v'$ are the endpoints of $\alpha_i$, and where $\operatorname{dist}(\cdot, \cdot)$ denotes the signed distance between horocycles. Further $\ell(\beta)$ remains constant along the earthquake path. Finally, for any edge $\alpha_k$ which shares one endpoint $v$ with $\beta$, we have that $\frac{d}{dt} \ell(\alpha_k) = \pm 1/2$ is independent of $p$ and $t$; the sign depends on whether $\alpha_k$ lies on one side of $\beta$, or the other.
\end{proof}

\subsection{Proof of Theorem~\ref{thm:main-HP}}
\label{sec:proof-main-HP}

We now have tools to prove Theorem~\ref{thm:main-HP}. First, however, we must prove Proposition~\ref{prop:imageinA-HP}.

\begin{proof}[Proof of Proposition~\ref{prop:imageinA-HP}]
We must prove that the dihedral angles $\theta = \Psi_{\Gamma}(P)$ of any ideal polyhedron $P \in \HPPoly_{\Gamma}$ satisfies the three conditions (i), (ii), (iii) required for $\theta \in \Angles_{\Gamma}$, as described in Section~\ref{sec:intro-AdS-param}. Condition (i) is simply our convention of labeling the dihedral angles of equatorial edges with negative signs. So, we must prove that $\theta$ satisfies (ii) and (iii).

That $\theta$ satisfies condition (ii) follows from the fact that the sum of the dihedral angles at a vertex of an ideal polyhedron in $\HH^3$ is constant (equal to $2\pi$). By Proposition~\ref{prop:angles-limit}, the dihedral angles of $P$ are simply the derivatives of the (exterior) dihedral angles of $Q_t$, where $Q_t$ is a path of ideal polyhedra in hyperbolic space (or anti-de Sitter space), as in Proposition~\ref{prop:limit}.

Now, let us prove that $\theta$ satisfies (iii). Consider a path $c$ on $P$ normal to the $1$--skeleton $\Gamma$ and crossing exactly two edges of the equator. Then, without affecting the combinatorics of the path, we deform so that $c$ is precisely $P \cap H$ for some vertical (degenerate) plane $H$ that is orthogonal to \emph{both} edges of the equator crossed by~$c$. Note that the angle between a non-degenerate line $\alpha$ and a degenerate plane $H$ is precisely the angle formed between the lines $\varpi(\alpha)$ and $\varpi(H)$ in $\HH^2$ and therefore we can indeed achieve that $H$ is orthogonal to both edges of the equator simultaneously (by contrast to the analogous situation in $\HH^3$ or $\AdS^3$).
The plane $H$ is isomorphic to a copy of two-dimensional half-pipe geometry $\HP^2$. Inside $H$, the edges of $c$ are non-degenerate, forming a polygon with exterior angles \emph{bounded above} by the corresponding dihedral angles of~$P$; indeed if $\theta_i$ is the dihedral angle between two faces in $\HP^3$ and $\vartheta_i$ is the angle formed by those faces when intersected with $H$, then by Proposition~\ref{prop:HP-angle-bound}, $\operatorname{sign}(\vartheta_i) = \operatorname{sign}(\theta_i)$ and $|\vartheta_i| \leq |\theta_i|$ with equality if and only if $H$ is orthogonal to the line of intersection between the faces; therefore the exterior angle in $H$ at each of the two points where $c$ intersects the equator is equal to the exterior dihedral angle along that equatorial edge (and both are negative) while the exterior angle at any other vertex of $c$ is strictly less that the exterior dihedral angle of $P$ at the corresponding edge (and both are positive). By the infinitesimal Gauss-Bonnet formula in $\HP^2$ (Proposition~\ref{prop:inf-GB}), the sum of the exterior angles of $c$ is positive and so it follows that the sum of the exterior dihedral angles over the edges of $P$ crossed by $c$ is also positive.
\end{proof}

\begin{proof}[Proof of Theorem~\ref{thm:main-HP}]
The map $F: \HPPoly \to \poly \times \Angles$, taking an HP ideal polyhedron to its projection to $\HH^2$, an ideal polygon, and to its dihedral angles, has a continuous left inverse. Let $G: \poly \times \Angles \to \HPPoly$ be the map that takes $p \in \poly$ and bends according to the \emph{top angles} $\theta_+$ of $\theta \in \Angles$, ignoring the rest of the information in $\theta$ (the bottom and equatorial dihedral angles). Then $G \circ F$ is the identity. Hence, to show that $\Psi = \PsiHP$ is a homeomorphism, we need only show that there is a continuous map $H: \Angles \to \poly$ such that $G(H(\Psi(P)), \Psi(P)) = P$. The existence of such a continuous map $H$ is guaranteed by Proposition~\ref{pr:convex} and a simple application of the Implicit Function Theorem as follows. For $\theta \in \Angles$, define $H(\theta)$ to be the unique minimum in $\poly$ of $\ell_\theta$ given by Proposition~\ref{pr:convex}. That $H$ is continuous (in fact differentiable on all strata of $\Angles$) follows from the convexity of $\ell_\theta$, thought of as a function on $\poly$. Now, recall that the space of ideal polygons $\poly$ identifies with the space of doubles $\doubles$ in $\Teich_{0,N}$. Hence, because $H(\theta)$ minimizes $\ell_\theta$ over $\poly$, the restriction of $d \ell_\theta$ (now thought of as a one-form on all of $\Teich_{0,N}$) to $\doubles$ is zero at (the double of) $H(\theta)$. 
It then follows that the infinitesimal shear $e_\theta$ on $\Teich_{0,N}$ is tangent to the subspace of doubles $\doubles$ at (the double of) $H(\theta)$ because $e_\theta$ is dual to $\ell_\theta$ (Lemma~\ref{lm:gradient}) and the space of doubles $\doubles$ is Lagrangian (Proposition~\ref{pr:doubles}).
Therefore $e_\theta$ determines a well-defined infinitesimal deformation of the polygon $H(\theta)$ and the pair $p = H(\theta), V = e_\theta(H(\theta))$ determines an HP polyhedron $P$ such that $F(P) = (H(\theta), \theta)$ as in the discussion in Section~\ref{subsec:polyhedra-HP}. The formula $G(H(\Psi(P)),\Psi(P)) = P$ follows, and this completes the proof of Theorem~\ref{thm:main-HP}.
\end{proof}

%
%
%

\section{Properness}\label{sec:properness}

In this section we will prove the two properness lemmas needed for the proofs of the main results. Lemma \ref{lem:Phi-proper} states that the map $\Phi$, sending an ideal polyhedron in $\AdS^3$ to its induced metric, is proper. Lemma \ref{lem:PsiAdS-proper}, when combined with Proposition~\ref{prop:imageinA-AdS}, will imply properness of the map sending a polyhedron of fixed combinatorics to its dihedral angles.

\subsection{Properness for the induced metric (Lemma \ref{lem:Phi-proper})}

To prove Lemma \ref{lem:Phi-proper}, we consider a compact subset $\mathcal K \subset \Teich_{0,N}$. 
We must show that the set $\Phi^{-1}(\mathcal K)$ is a compact subset of $\overline \AdSPoly$. 
In other words, if $P$ is a polyhedron with $m = \Phi(P) \in \mathcal K$, 
we must show that $P$ lies in a compact subset of $\overline \AdSPoly$.

Since there are finitely many triangulations of the disk with $N$ vertices, 
we may consider polyhedra $P$ with fixed combinatorics, that is
the graph $\Gamma$ is fixed. We may assume $\Gamma$ is a traingulation by adding edges if necessary.

\begin{figure}[hbt] 
\centering 
\includegraphics[height=5cm]{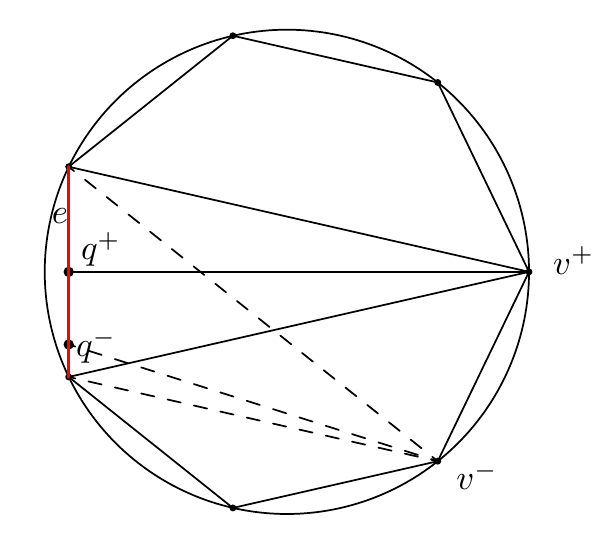}
\caption{The polyhedron $P$ with combinatorics given by $\Gamma$. 
The red edge is $e$, and $q^+ = \pi_R(v^+; e)$, $q^- = \pi_R(v^-; e)$.}
\label{Fig:hep}
\end{figure}

Recall that the induced metric $m$ on $P$ is related to the left and right metrics 
$m_L$ and $m_R$ by the diagram in Theorem \ref{thm:earthquakes}: $m_R = E_\theta(m)$ and $m = E_\theta(m_L)$, where $\theta: \Gamma \to \RR$ is the assignment of  exterior dihedral angles to the edges of $P$ and $E_\theta$ is the shear amp associated to $\theta$. 
Also, recall that $m_L$ and $m_R$ are cusped metrics on the sphere that come 
from doubling the metric structures on the ideal polygon 
obtained by projecting the vertices of $P$ to
the left and right foliations of $\partial_\infty \AdS^3$.
To show that $P$ lies in a compact set, we must show that $m_L$ and $m_R$ 
lie in compact sets. It is enough to show that 
$\theta$ remains bounded over $\Phi_\gamma^{-1}(\mathcal K)$.
Although we have not yet proved Proposition~\ref{prop:imageinA-AdS}, we will use here that $\theta$ satisfies conditions (i) and (ii) in the definition of $\Angles_\Gamma$ (Section~\ref{sec:intro-AdS-param}). That these conditions are satisfied is essentially trivial, see Section~\ref{sec:necess}.

Consider an edge $e$ of the equator $\gamma$ of $\Gamma$, and recall that $\theta(e) < 0$ (condition (i) of the definition of $\Angles_\Gamma$). Let $s_L(e), s_R(e), s(e)$ denote the shear coordinate along $e$, with respect to $\Gamma$, 
of the left metric $m_L$, the right metric $m_R$, and the induced metric $m$. 
Then, by Theorem \ref{thm:earthquakes}, we have:
$$ s_R(e) - s(e) = \theta(e) = s(e) - s_L(e)~.$$
Now the edge $e$ belongs to a unique triangle of $\Gamma$ in the top hemisphere of $\Sigma_{0,N}$, 
the third vertex of which we denote by $v^+$. 
On the bottom hemisphere, the edge $e$, again, belongs to a unique triangle, whose third vertex we denote by $v^-$.

There are two cases to consider. Recall that we fixed an orientation of the equator $\gamma$. Imagining that we view $\Sigma_{0,N}$ from above, it is intuitive to call the positive direction \emph{left} and the negative direction \emph{right}. 
First suppose $v^+$ lies to the left of $v^-$ when viewed from $e$. 
The restriction of the right metric $m_R$ to the top hemisphere of $\Sigma_{0,N}$ is a marked hyperbolic ideal polygon $p_R$, in which the vertex $v^+$ again lies to the left of $v^-$. 
Since $m_R$ is the double of $p_R$, 
we may calculate the shear coordinate $s_R(e)$ of $s_R$ by the simple formula:
$$ s_R(e) = \pi_R(v^+; e) - \pi_R(v^-; e),$$
where $\pi_R(v; e)$ denotes the projection of $v$ onto the edge $e$ in $p_R$, see Figure \ref{Fig:hep}. Then we have 
$s_R(e) > 0$ and so $s_L(e) > s(e) > s_R(e) > 0$. In particular,
$$ \theta(e) = s_R(e) - s(e) > -s(e).$$
In the case that $v^+$ lies to the right of $v^-$, we examine the left metric $m_L$. 
In the restriction $p_L$ of $m_L$ to the top hemisphere, the vertex $v^+$ again lies to the right of $v^-$ 
and so, by a similar calculation as above, $s_L(e) < 0$ and so $s_R(e) < s(e) < s_L(e) < 0$. Therefore
$$ \theta(e) = s(e) - s_L(e) > s(e)~.$$
In either case, $\theta(e)$ is bounded, because the shears $s(e)$ are bounded, 
as $m$ varies over the compact set $\mathcal{K}$. 

We have shown that all of the edges $e$ for which $\theta(e) < 0$ have $\theta(e)$ bounded. 
It then follows that the other edges $e'$, for which $\theta(e') \geq 0$, 
also have $\theta(e')$ bounded, since the sum of all positive and negative angles along edges 
coming into any vertex of $P$ must be zero (condition (ii) of the definition of $\Angles_\Gamma$). Therefore $\Phi^{-1}(\mathcal K)$ is compact.

\subsection{Proof of Lemma \ref{lem:PsiAdS-proper}}

Let $\Gamma \in \Graph(\Sigma_{0,N}, \gamma)$. We consider a sequence $(P_n)_{n\in \N}$ going to infinity in $\AdSPoly_N$ such that the dihedral angles $\theta_n = \PsiAdS_\Gamma(P_n)$ converge to $\theta_\infty \in \RR^E$, where $E = E(\Gamma)$ denotes the edges of $\Gamma$ as usual.
We msut show that $\theta_\infty$ fails to satisfy condition (iii) of Section~\ref{sec:intro-AdS-param}.

For each $n$, let $p^L_n = \varpi_L(P_n)$ and $p^R_n = \varpi_R(P_n)$ be the ideal polygons whose ideal vertices are the left and right projections of the ideal vertices of $P_n$ (as in Section~\ref{sec:ideal_poly}).
 Let $v^L_{1,n},\cdots, v^L_{N,n}$ denote the vertices in $\RP^1$ of $p_n^L$, and similarly 
let $v^R_{1,n},\cdots, v^R_{N,n}$ denote the vertices of $p_n^R$. By applying an isometry of $\AdS^3$, 
we may assume that the first three vertices of $P_n$ are $(0,0)$, $(1,1)$ and $(\infty, \infty)$ independent of $n$, so that
$v^L_{1,n} = v^R_{1,n} = 0$, $v^L_{2,n} = v^R_{2,n} = 1$ and $v^L_{3,n} = v^R_{3,n} = \infty$ for all $n$.

Since $\theta_n$ converges to the limit $\theta_\infty$ and the polyhedra 
$P_n$ diverge, the sequence of
ideal polygons $(p_n^L)_{n\in \N}$ diverges (in the space of ideal $N$-gons up to equivalence). 
Reducing to a subsequence, we may assume all of the vertices converge to well-defined 
limits $v^L_{i,n} \to v^L_{i, \infty} \in \mathbb{RP}^1$. 
However, for some indices $i$, we will have $v^L_{i, \infty} = v^L_{i+1, \infty}$. 
Now, since the right polygon $p^R_n$ is obtained from $p^L_n$ by an earthquake of bounded magnitude,
it follows that each vertex $v^R_{i,n}$ also converges to a well-defined limit $v^R_{i, \infty}$ 
and that $v^L_{i, \infty} = v^L_{i+1, \infty}$ if and only if $v^R_{i, \infty} = v^R_{i+1, \infty}$. 
In other words the polyhedra $P_n$ converge to a convex ideal polyhedron $P_\infty$ 
of strictly fewer vertices. 

The combinatorial structure of $P_\infty$ is obtained from $\Gamma$ by collapsing 
vertices and the corresponding edges and faces in the obvious way: 
if two vertices that span an edge collapse together, then that edge disappears. 
If that edge bounded a triangle, then that triangle collapses to an edge, and so on. 
Let $\Gamma_\infty$ denote the $1$-skeleton of $P_\infty$, and let $\Gamma_\infty^*$ denote the dual graph.  
Consider a simple path 
$c_\infty$ in $\Gamma_\infty^*$. We may lift $c_\infty$ to a path $c$ in the dual graph $\Gamma^*$ 
fof $\Gamma$ in the obvious way: an edge of $c_\infty$ is dual to an edge $e$ of $P_\infty$. 
Under the collapse $\Gamma \to \Gamma_\infty$, $e$ lifts to a collection of consecutive edges 
in $\Gamma$ which determines a path of adjacent edges in $\Gamma^*$.   
The sum of the dihedral angles assigned by $\theta_n$ 
to the path $c$ converges to the sum of the dihedral angles of $P_\infty$ over the edges of $c_\infty$.

Now consider an ideal vertex of $P_\infty$ which is the limit of two or more vertices of the $P_n$, 
and let $c_\infty$ denote the path of edges bounding the face of $\Gamma_\infty^*$ dual to this vertex. 
Of course, the sum of the angles over the edges of $c_\infty$ is zero (by condition (ii) in the definition of $\Angles_{\Gamma_\infty}$). 
It therefore follows that $\theta_\infty$ assigns angles that sum to zero around the edges of the path $c$. 
Therefore $\theta_\infty$ violates condition (iii) in the definition of $\Angles_\Gamma$, since $c$ does not bound a face in $\Gamma^*$, and the proof is complete.
\section{Rigidity}\label{sec:rigidity}
This section is dedicated to the local versions of Theorems~\ref{thm:main-AdS-angles} and ~\ref{thm:main-AdS-metrics}, which are Lemma~\ref{lem:Phi-rigidity} and Lemma~\ref{lem:PsiAdS-rigidity}.

\subsection{The Pogorelov map for $\AdS^n$}\label{sec:Pogorelov}

We recall here the definition and main properties of the infinitesimal Pogorelov map,
which turns infinitesimal rigidity problems for polyhedra (or submanifolds) in constant 
curvature pseudo-Riemannian space-forms into similar infinitesimal rigidity problems 
in flat spaces, where they are easier to deal with. 
These maps, as well as their non-infinitesimal counterparts, were discovered by Pogorelov 
\cite[Chapter 4]{Pogorelov} (in the Riemannian case). Another account and some geometric
explanations of the existence of these maps can be found in Schlenker \cite[Prop. 5.7]{shu} or in
Fillastre \cite[Section 3.3]{fillastre3}. See also Labourie--Schlenker \cite[Cor. 3.3]{iie} or Izmestiev \cite{izmestiev:projective}. We follow here mostly the presentation given in \cite[Section 3.3]{fillastre3}, and refer to this paper for the proofs.

Although we will return to dimension three shortly, we describe the Pogorelov map in any dimension~$n$.
Consider the complement $U$ in $\AdS^n$ of a spacelike totally geodesic hyperplane $H_0$, dual to a point $x_0 \in \AdS^n$. Here duality means that $H_0$ is defined by the equation $\langle x_0, x\rangle = 0$, where $\langle \cdot, \cdot \rangle$ is the inner product of signature $(n-1,1)$ defining $\AdS^n$. Then $U$ is naturally the intersection of $\AdS^n$ with an affine chart $\RR^n$ of projective space, and we may take $\iota(x_0) = 0$ to be the origin of this affine chart, where $\iota : U \hookrightarrow \RR^n$ denotes the inclusion. The union of all light-like geodesics passing through $x_0$ is called the light cone $C(x_0)$. 

We equip $\RR^n$ with a flat Lorentzian metric, making it into a copy of Minkowski space $\RR^{n-1,1}$. We may choose this metric so that the inclusion $\iota$ is an isometry at the tangent space to~$x_0$. 
This has the pleasant consequence that $\iota(C(x_0))$ is precisely the light cone of $\iota(x_0)$ in $\RR^{n-1,1}$. We now define a bundle map $\Upsilon: TU \to T \RR^{n-1,1}$ over the inclusion $\iota : U \hookrightarrow \RR^{n-1,1}$ as follows: $\Upsilon$ agrees with $d \iota$ on $T_{x_0} U$. For any $x \in U \setminus C(x_0)$, and any vector $v \in T_x  U$, write $v = v_r + v_{\perp}$, where $v_r$ is tangent to the radial geodesic passing through $x_0$ and $x$, and $v_{\perp}$ is orthogonal to this radial geodesic, and define
\begin{equation}\label{eqn:Pogorelov}
\Upsilon(v) = \sqrt{\frac{\| \hat r \|^2}{\| d\iota(\hat r) \|^2}} d\iota(v_r) + d\iota( v_{\perp}),
\end{equation}
where $\hat r$ is the unit radial vector (so $\|\hat r\|^2 = \pm 1$) and the norm $\| \cdot \|$ in the numerator of the first term is the AdS metric, while the norm in the denominator is the Minkowski metric. Note that a radial geodesic of $U$ (passing through $x_0$) is taken by $\iota$ to a radial geodesic in $\RR^{n-1,1}$ (passing through the origin) of the same type (space-like, light-like, time-like), although the length measure along the geodesic is not preserved. Hence the quantity under the square-root in \eqref{eqn:Pogorelov} is always positive.

The key property of the infinitesimal Pogorelov map is the following (the proof is an easy computation in coordinates, see \cite[Lemma 3.4]{fillastre3}).

\begin{lemma} \label{lm:pogorelov}
Let $Z$ be a vector field on $U \setminus C(x_0) \subset \AdS^n$. Then $Z$ is a Killing vector field if and only
if $\Upsilon(Z)$ (wherever defined) is a Killing vector field for the Minkowski metric on $\R^{n-1,1}$.
\end{lemma}

\noindent In fact, the lemma implies that the bundle map $\Upsilon$, which so far has
only been defined over $U \setminus C(x_0)$, has a continuous extension to all of~$U$. The bundle map $\Upsilon$ is called an \emph{infinitesimal Pogorelov map}.

Next, the bundle map $\Xi: T\R^{n-1,1}\to T\R^{n}$ over the identity, which simply
changes the sign of the $n$-th coordinate of a given tangent vector,
has the same property: it sends Killing vector fields in $\R^{n-1,1}$ to Killing vector
fields for the Euclidean metric on $\R^{n}$. Hence the map $\Pi = \Xi \circ \Upsilon$ is a bundle map over the inclusion $U \hookrightarrow \RR^n$ with the following property:

\begin{lemma} \label{lm:pogorelov-Euclidean}
Let $Z$ be a vector field on $U \subset \AdS^n$. Then $Z$ is a Killing vector field if and only
if $\Pi(Z)$ is a Killing vector field for the Euclidean metric on $\RR^n$.
\end{lemma}

\noindent The bundle map $\Pi$ is also called an infinitesimal Pogorelov map. Henceforth we return to the setting of three-dimensional geometry.

\subsection{Rigidity of Euclidean polyhedra}

In order to make use of the infinitesimal Pogorelov map defined above, we recall
some elementary and well-known results about the rigidity of convex Euclidean polyhedra. 
It has been known since Legendre \cite{legendre} and Cauchy \cite{Cauchy}
that convex polyhedra in Euclidean three-space $\RR^3$ are globally rigid. In fact, given two polyhedra $P_1, P_2$, if there is map $\partial P_1 \to \partial P_2$ which respects the combinatorics and is an isometry on each face, then the map is the restriction of a global isometry of Euclidean space. 
Later Dehn \cite{dehn-konvexer} proved that convex Euclidean polyhedra are also 
infinitesimally rigid. In fact, he showed that any first-order deformation $V$ of a polyhedron $P$ that preserves the combinatorics and the metric on each face is the restriction of a global Killing vector field. Here $V$ is not allowed, for example, to deform the polyhedron so that a quadrilateral face becomes two triangular faces. Still later, A.D. Alexandrov  \cite{alexandrov} proved a stronger version of this statement:

\begin{theorem}[Alexandrov] \label{thm:alexandrov}
Let $P$ be a convex polyhedron in $\R^3$, and let $V$ be an infinitesimal deformation
of $P$ (that might or might not change the combinatorics). Then, if the induced metric on each face is fixed, at first order, under
 $V$, the deformation $V$ is the restriction to $P$ of a global Euclidean Killing field.
\end{theorem}

\subsection{Proof of Lemma~\ref{lem:Phi-rigidity} (and Lemma~\ref{lem:PsiAdS-rigidity})}

We first prove Lemma~\ref{lem:Phi-rigidity}. Lemma~\ref{lem:PsiAdS-rigidity} then follows from it and Theorem~\ref{thm:duality}.

Let $P\in \AdSPoly_N$. We argue by contradiction and
suppose that $\Phi$ is not a local immersion at $P$. This means
that there exists a tangent vector $V$ to $\AdSPoly$ at $P$ such that 
$d\Phi(V)=0$. In other terms, there is a first-order deformation $V$ of 
$P$, as an ideal polyhedron in $\AdS^3$, which does not change the induced metric.

Now, $V$ is described by tangent vectors $V_i \in T_{z_i} \partial_\infty \AdS^3$ at each ideal vertex $z_i$. Since $P$ is convex, it is contained in the complement $U \subset \AdS^3$ of a spacelike totally geodesic plane. We wish to use the Pogorelov map $\Pi$ defined in Section~\ref{sec:Pogorelov} above. However, $\Pi$ is \emph{not} defined over the ideal boundary, so we need to be slightly careful.
We may assume that the $1$--skeleton $\Gamma$ of $P$ is a triangulation. If not, we simply add diagonals to all of the non-triangular faces as needed. Consider a triangular face $T = \Delta z_{i_1} z_{i_2} z_{i_3}$. The tangent vectors $V_{i_1}, V_{i_2}, V_{i_3}$ determine a unique Killing field $X$, which defines the motion of the points of $T$ under the deformation. The deformation vectors for the vertices of an adjacent triangle $T' = \Delta z_{i_2} z_{i_1} z_{i_4}$ similarly determine a Killing field $X'$, which determines the motion of the points of $T'$. In general, $X$ and $X'$ might not agree on the common edge $e = z_1 z_2$. However, because $d\Phi(V) = 0$, the shear coordinate along $e$ does not change to first order, and therefore $X$ and $X'$ do agree along the edge $e$. It follows that $V$ defines a vector field $W$ on $\partial P$ whose restriction to any face agrees with a Killing field of $\AdS^3$. We now apply the Pogorelov map to obtain $\Pi(W)$, a vector field on the boundary of a convex polyhedron $\iota(P)$ in Euclidean space $\RR^3$. By Lemma~\ref{lm:pogorelov-Euclidean}, the restriction of $\Pi(W)$ to each face of $\iota(P)$ agrees with a Euclidean Killing field. By Theorem~\ref{thm:alexandrov}, $\Pi(W)$ must be the restriction of a global Euclidean Killing field~$Y$. Hence, again, using Lemma~\ref{lm:pogorelov-Euclidean}, we see that $W$ was the restriction of a global Killing field $\Pi^{-1}(Y)$ of $\AdS^3$ and therefore $V$ represents the trivial deformation in $\AdSPoly_N$. This completes the proof of Lemma~\ref{lem:Phi-rigidity}.

%
%
%

\section{Necessary conditions on the dihedral angles: proof of Proposition~\ref{prop:imageinA-AdS}}\label{sec:necess}

In this section we prove Proposition~\ref{prop:imageinA-AdS}, which states that the map $\PsiAdS_\Gamma$, taking an ideal polyhedron $P$ in $\AdS$ with $1$--skeleton $\Gamma$ to its dihedral angles $\theta = \PsiAdS_\Gamma(P)$, has image in the convex cone $\Angles_\Gamma$; in other words $\theta$ satisfies conditions (i), (ii), and (iii) of Section~\ref{sec:intro-AdS-param}. That $\theta$ satisfies (i) is just our sign convention for dihedral angles. That the dihedral angles $\theta$ satisfy (ii) follows exactly as in the hyperbolic setting: The intersection of $P$ with a small ``horo-torus" centered about an ideal vertex of $P$ is a convex polygon in the Minkowski plane, whose exterior angles are equal to the corresponding exterior dihedral angles on $P$.

The difficult part of Proposition~\ref{prop:imageinA-AdS} is to prove that $\theta$ satisfies condition~(iii), and the remainder of this section is dedicated to this claim. Consider a simple cycle $e_0^*, e_1^*, \ldots, e_n^* = e_0^*$ in $\Gamma^*$ such that $\theta(e_j^*) < 0$ for exactly two edges $j = 1, r$. Let $f_i^*$ be the vertex of $\Gamma^*$, dual to a two-dimensional face $f_i$ of $P$, which is an endpoint of $e_i^*$ and $e_{i+1}^*$. In other words, the face $f_i$ of $P$ contains the edges $e_i$ and $e_{i+1}$. We must prove that the sum $\theta(e_1^*) + \cdots + \theta(e_n^*) > 0$.

We now define a polyhedron $Q$ by ``extending" the faces $f_1, \ldots, f_n$ and forgetting about the other faces of $P$. More rigorously: Since $P$ is contained in an affine chart of $\RP^3$, a lift $\tilde P$ of $P$ to the three-sphere $S^3$ is a convex polyhedron contained in an open half-space of $S^3$. Define $\tilde Q$ to be the intersection of the half-spaces defined by the lifts of $f_1, \ldots, f_n$. Then generically $\tilde Q$ will be contained in an open half-space, in which case $\tilde Q$ projects to a compact polyhedron $Q$ in some affine chart of $\RP^3$. We will, in a sense, reduce the generic case to the easier case that $\tilde Q$ is \emph{not} contained in an open half-space, which we treat first. In this case, the combinatorial structure of $\tilde Q$ is very simple in that $\tilde Q$ has exactly two antipodal vertices. The projection $Q$ of $\tilde Q$ to $\RP^3$ has one vertex, which is contained in every face $f_1, \ldots, f_n$ and edge $e_1, \ldots, e_n$ of $Q$. Therefore $f_1, \ldots, f_n$ and $e_1, \ldots, e_n$ are orthogonal to the time-like plane $A$ dual to that vertex. As in Section~\ref{sec:Pogorelov}, duality is defined with respect to the inner product of signature $(2,2)$ that defines $\AdS^3$. The intersection $q = A \cap Q$ is a convex compact polygon lying in $A \cong \AdS^2$ whose exterior angles are equal to the exterior dihedral angles of $Q$. That (iii) holds in this case now follows from:

\begin{claim} \label{claim:AdS2-angles}
The sum of the exterior angles of a compact, convex, space-like polygon $q$ in $\AdS^2$ is strictly positive.
\end{claim}

\begin{proof}
This follows directly from the Gauss-Bonnet formula for Lorentzian polygons (see \cite{bir_gau}). Alternatively, one may easily prove the claim directly for triangles and then argue by induction. 
\end{proof}

Before continuing to the general case, it is useful to examine the dual picture in this simple case. Let $q^*$ denote the dual convex polygon in $\RP^2$ (where $\RP^2$ is identified with its dual via the signature $(2,1)$ inner product defining $\AdS^2$). Since all edges of $q$ are space-like, the vertices of $q^*$ are contained in $\AdS^2$.  If $v$ is a vertex of $q$ with positive exterior angle, then the dual edge $v^*$ in $q^*$ is a space-like edge contained in $\AdS^2$. However, if $v$ is a vertex of $q$ with negative exterior angle, then the dual edge $v^*$ of $q^*$ is begins and ends in $\AdS^2$ but contains a segment outside of $\AdS^2$. Conversely, the dual of any convex polygon $q^*$ in $\RP^2$ having the properties just described is a convex compact polygon in $\AdS^2$.  Note that the length of an edge in $q^*$ is equal to the dihedral angle at the corresponding vertex of $q$, with the two edges which leave $\AdS^2$ having negative length. Therefore Claim~\ref{claim:AdS2-angles} is equivalent to:

\begin{claim} \label{claim:AdS2-dual-lengths}
Let $q^*$ be a convex polygon in $\RP^2$ with vertices in $\AdS^2$ and with space-like edges, all but two (non-adjacent) of which are contained in $\AdS^2$. Then the sum of the lengths of the edges of $q^*$ is positive.
\end{claim}

\noindent This dual point of view will be useful in the general case, which we turn to now.

Consider the generic case that the polyhedron $Q$ is compact in an affine chart of $\RP^3$. In this case, $Q$ will have extra edges, in addition to $e_1, \ldots, e_n$, which are not contained in $\AdS^3$; these edges may be either space-like or time-like. Let $Q^*$ denote the dual polyhedron in $\RP^3$, where we identify $\RP^3$ with its dual via the inner product of signature $(2,2)$ that defines $\AdS^3$. By perturbing a small amount if necessary, we may assume that all vertices of $Q$ lie outside of the closure of $\AdS^3$, so that the faces of $Q^*$ are each time-like. The vertices of $Q^*$, dual to the space-like faces $f_1, \ldots, f_n$, lie in $\AdS^3$.  The dual edges $e_1^*, \ldots, e_n^*$ are space-like and form a Hamiltonian cycle in $\partial Q^*$ dividing it into two convex polyhedral surfaces $(\partial Q^*)_1$ and~$(\partial Q^*)_2$. We need only work with one of these surfaces, say $(\partial Q^*)_1$. The surface $(\partial Q^*)_1$ is a polygon, bent along some interior edges. Note that two of the perimeter edges $e_1^*$ and $e_r^*$ of $(\partial Q^*)_1$ each contain a segment outside of $\AdS^3$, while $e_2^*, \ldots, e_{r-1}^*$ and $e_{r+1}^*, \ldots, e_n^*$ are contained in $\AdS^3$.  We will show:
\begin{lemma}\label{lem:intrinsically-convex}
The surface $(\partial Q^*)_1$ is intrinsically locally convex.
\end{lemma} 
\noindent The lemma says that when $(\partial Q^*)_1$ is ``un-folded" onto a time-like plane (a copy of $\AdS^2$), it is convex and therefore isomorphic to some $q^*$ as in Claim~\ref{claim:AdS2-dual-lengths} above. Therefore condition (iii) will follow from the lemma. 
Before embarking on the proof, we draw on some intuition from the Riemannian setting.
To show that a developable polyhedral surface $S$ in a Riemannian space ($\RR^3$ say) is intrinsically locally convex, one must simply show that the total angle of $S$ at each vertex is less than~$\pi$. Equivalently, one examines the \emph{link} of each vertex $v$ of $S$, which is naturally a polygonal path in the unit sphere in the tangent space at $v$: $S$ is locally convex at $v$ if and only if the length of this polygonal path is less than~$\pi$. We show that $(\partial Q^*)_1$ is locally convex in much the same way, by examining the link of each vertex of $(\partial Q^*)_1$ and measuring how long it is. However the space of rays emanating from a point in a Lorentzian space is not the Riemannian unit sphere, but rather what is called the \emph{HS sphere} or $\HS^2$.

\subsection{The geometry of the HS sphere}

\emph{HS geometry}, introduced in \cite{shu,cpt} and used recently in \cite{colI,colII}, 
is the natural local geometry near a point in a Lorentzian space-time such as $\AdS^3$. 
In those papers, HS-structures with cone singularities occur naturally as the induced geometric
structures on the boundary of polyhedra or, in a related manner, on the links of vertices of the singular
graph in Lorentzian 3-manifolds with cone singularities. Here we will use comparatively simpler notions
without cone singularities.

The tangent space at a point of $\AdS^3$ is a copy of the three-dimensional Minkowski space $\RR^{2,1}$. 
The \emph{HS sphere} $\HS^2$ is the space of rays based at the origin in $\RR^{2,1}$.
It admits a natural decomposition into five subsets:
\begin{itemize}
  \item Let $\HH^2_+$ (respectively $\HH^2_{-}$) denote the future oriented (resp. past oriented) time-like rays. Both $\HH^2_+$ and $\HH^2_-$ are copies of the Klein model for the hyperbolic plane and are equipped with the standard hyperbolic metric in the usual way. 
    \item Let $\dS^2$ denote the space-like rays, equipped with the standard de Sitter metric.
  \item The light-like rays form two circles, $\partial\HH^2_+$ and $\partial\HH^2_{-}$, which are the boundaries of $\HH^2_+$ and $\HH^2_{-}$ respectively. 
\end{itemize} 
\noindent The group $\operatorname{SO}_0(2, 1)$ of time-orientation and orientation preserving linear isometries of $\R^{2,1}$ acts naturally (and projectively) on $\HS^2$, preserving this decomposition. The geodesic $\sigma_{x,y}$ between two (non-antipodal) points $x,y \in \HS^2$ is defined to be the positive span of the two rays $x,y$.
The space $\HS^2$ is equipped with a (partially defined) signed distance function $d(\cdot, \cdot)$ as follows. 
\begin{itemize}
\item If $x, y \in \HH^2_+$ or $x, y \in \HH^2_-$ then $d(x,y)$ is the usual hyperbolic distance, equal to the hyperbolic length of $\sigma_{x,y}$.
\item Let $x, y \in \dS^2$. We will only be interested in the case that $\sigma_{x,y}$ is time-like, meaning the plane in $\RR^{2,1}$ spanned by $\sigma_{x,y}$ has mixed signature. If $\sigma_{x,y}$ is contained in $\dS^2$, then $d(x,y)$ is defined to be the de Sitter length of $\sigma_{x,y}$, taken to be a negative (rather than imaginary) number. Note that in this case $d(x,y) = - d(x^*, y^*)$, where $x^*$ (resp. $y^*$) denotes the geodesic line dual to $x$ (resp. $y$) in $\HH^2_+$ (equal to the intersection with $\HH^2_+$ of the orthogonal complement of $x$ (resp. $y$)). In the case that $\sigma_{x,y}$ passes through $\HH^2_+$ (or $\HH^2_-$), we define $d(x,y) = +d(x^*, y^*)$.
\item Let $x \in \HH^2_+$ and $y \in \dS^2$. Then we define $d(x,y) = +d(x,y^*)$ if $x$ and $y$ lie on opposite sides of $y^*$ or $d(x,y) = -d(x,y^*)$ if $x$ and $y$ lie on the same side of $y^*$.
\end{itemize}

\noindent We note that this distance function may be similarly defined in terms of the Hilbert distance (cross ratios) with respect to $\partial \HH^2_+$ or $\partial \HH^2_-$.
Let $\sigma$ be a polygonal path in $\HS^2$ with endpoints $x,y \in \dS^2$ and call $\sigma$ \emph{time-oriented} if $\sigma$ is the concatenation of three polygonal subpaths: a path crossing from from $x$ to $\HH^2_+$ which is future-oriented, followed by a path in $\HH^2_+$, followed by a path from $\HH^2_+$ back into $\dS^2$ which is past oriented. The length $\mathscr L(\sigma)$ is defined to be the sum of the lengths of the geodesic segments comprising $\sigma$. It is important to note that $\mathscr L(\sigma)$ is well-defined under sub-division. 
The crucial ingredient in the proof of Lemma~\ref{lem:intrinsically-convex} is the following substitute for the triangle inequality.
\begin{claim} \label{cl:HS}
Let $\sigma$ be a time-oriented polygonal path with endpoints $x,y \in \dS^2$ and suppose further that $\sigma_{x,y}$ is time-like and crosses through $\HH^2_+$. Then $\mathscr L(\sigma) \geq \mathscr L(\sigma_{x,y})$.
\end{claim}

\begin{proof}
Let $x= x_0, x_1,\ldots, x_n = y$ be the ordered vertices of $\sigma$, with $x_0, \ldots, x_i$ lying in $\dS^2$, $x_{i+1}, \ldots, x_j$ lying in $\HH^2_+$, and $x_{j+1}, \ldots, x_n$ lying in $\dS^2$. Then,
\begin{align*} \mathscr L(\sigma) &= -\sum_{k=0}^{i-1} d(x_k^*, x_{k+1}^*) \ \ \ +\epsilon_1 d_(x_i^*, x_{i+1})\\ & + \sum_{k=i+1}^j d(x_{k}, x_{k+1}) \ \ \ +\epsilon_2 d_(x_j, x_{j+1}^*) 
-\sum_{k = j+1}^{n-1} d(x_{k}^*, x_{k+1}^*)
\end{align*}
where $\epsilon_1, \epsilon_2 = \pm 1$.
We may assume, by sub-dividing, that $x_i$ and $x_{i+1}$ are on the same side of $x_i^*$ and that $x_j$ and $x_{j+1}$ are on the same side of $x_{j+1}^*$, so that $\epsilon_1 = \epsilon_2 = -1$. Therefore all of the dual lines $x_0^*, \ldots, x_i^*$ and $x_{j+1}^*, \ldots, x_n^*$ lie in between $x_{i+1}$ and $x_j$ in $\HH^2_+$.
In fact, the dual lines are arranged, in order from closest to $x_{i+1}$ to closest to $x_j$, as follows: $x_i^*, x_{i-1}^*, \ldots, x_0^*, x_n^*, x_{n-1}^*, \ldots, x_{j+1}^*$. 
See Figure \ref{Fig:HS}.
\begin{figure}[hbt] 
{\centering 

 \def\svgwidth{4.5in}
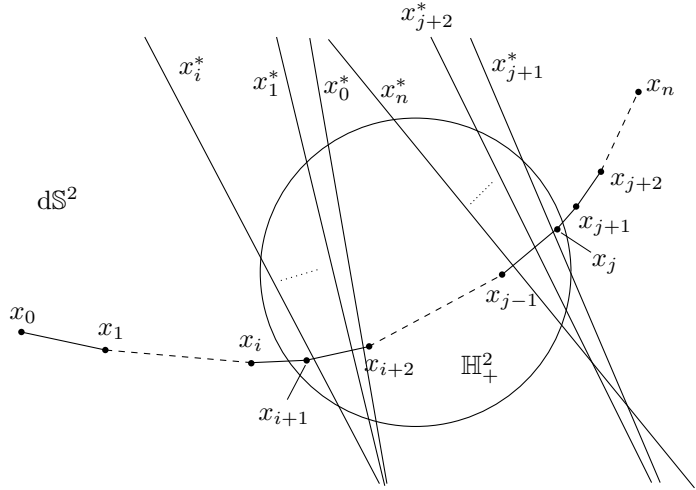

}

\caption{In the proof of Claim~\ref{cl:HS}, $\sigma$ is a time oriented polygonal path in $\HS^2$. In Lemma~\ref{lem:intrinsically-convex}, we apply Claim~\ref{cl:HS} to the case that $\sigma$ is the link of a vertex of $(\partial Q^*)_1$, which is convex (as drawn).}
\label{Fig:HS}
\end{figure}
Therefore, we have, by the triangle inequality in $\HH^2$, that 
\begin{align*}
d(x_{i+1}, x_{j}) \geq & \ \ d(x_{i+1}, x_i^*) + \sum_{k=1}^i d(x_k^*, x_{k-1}^*) \\ & + d(x_0^*, x_n^*) + \sum_{k=j+2}^ {n} d(x_k^*, x_{k-1}^*) + d(x_{j+1}^*, x_j)\end{align*}
since the line connecting $x_{i+1}$ to $x_j$ crosses each of the dual lines in the above equation. 
Again by the triangle inequality in $\HH^2$, we also have
$$d(x_{i+1}, x_{j}) \leq \sum_{k=i+1}^{j-1} d(x_{k}, x_{k+1}) .$$
It follows that $\mathscr L(\sigma_{x,y}) = d(x_0^*, x_n^*) \leq \mathscr L(\sigma)$.

\end{proof}

\subsection{Proof of Lemma~\ref{lem:intrinsically-convex}}

To complete the proof of Proposition~\ref{prop:imageinA-AdS}, we now prove Lemma~\ref{lem:intrinsically-convex} which states that the convex pleated polygon $(\partial Q^*)_1$ is intrinsically locally convex.
Consider a vertex $f_i^*$ of $(\partial Q^*)_1$.  We consider the link $\sigma$ at $f_i^*$ of $(\partial Q^*)_1$, a polygonal path in the space of rays in $T_{f_i^*} \AdS^3$ which is naturally a copy of $\HS^2$. The endpoints $x$ and $y$ of $\sigma$ correspond to two consecutive dual edges $e_i^*$ and $e_{i+1}^*$ in the perimeter of $(\partial Q^*)_1$. Since the $e_j^*$ are space-like, $x,y$ lie in $\dS^2 \subset \HS^2$. By assumption, the edges $e_i$ and $e_{i+1}$ intersect outside of $\AdS^3$ (at a point which is positive with respect to the $(2,2)$ form), and therefore the plane containing $e_i^*$ and $e_{i+1}^*$ (which is dual to this point) is time-like, thus so is $\sigma_{x,y}$. By convexity of $Q$, the geodesic $\sigma_{x,y}$ passes through a hyperbolic region of $\HS^2$, which without loss in generality we take to be $\HH^2_+$. Further, by convexity of $Q$ and the fact that each of the faces of $\partial Q^*$ is time-like, the link $\sigma$ at $f_i^*$ of $(\partial Q^*)_1$ is time-oriented in the sense defined in the previous section. Therefore, it follows from Claim~\ref{cl:HS} that $\mathscr L(\sigma) \geq \mathscr L(\sigma_{x,y}) > 0$.
Lemma~\ref{lem:intrinsically-convex} now follows because $\mathscr L(\sigma)$ is a complete invariant of the local geometry of $(\partial Q^*)_1$ at $f_i^*$; the development of $(\partial Q^*)_1$ onto a copy of $\AdS^2$ is convex at this vertex if and only if the length of the link is positive. This completes the proof of Proposition~\ref{prop:imageinA-AdS}.

%
%
%

\section{Topological arguments}\label{sec:topology}

\subsection{The topology of the space of ideal polyhedra}\label{sec:topolo_poly}

\begin{prop}\label{topol_poly}
If $N \geq 3$, the space $\overline{\AdSPoly}_N$ is connected.
If $N \geq 6$, then $\AdSPoly_N$ is connected and simply connected.
\end{prop}

\begin{proof}
By Proposition~\ref{prop:cyclic-order}, the space $\overline{\AdSPoly}_N$ identifies with the space $\poly_N \times \poly_N$ of pairs $(p_L, p_R)$ of marked ideal $N$-gons in the hyperbolic plane considered up to the action of $\PSL_2 \RR \times \PSL_2\RR$. 
The space $\AdSPoly_N$ is obtained from $\overline{\AdSPoly}_N$ by removing all pairs $(p_L, p_R)$ such that $p_L$ and $p_R$ are isometric.  Using the action of $\PSL_2 \RR \times \PSL_2\RR$ we may, in a unique way, put $p_L$ and $p_R$ into standard position so that the first three vertices of each polygon are $\infty,0$ and~$1$.  The remaining vertices of $p_L$ form an increasing sequence of $N-3$ points $x_4<\cdots < x_N$ in $(1,\infty)$. Similarly, the remaining vertices of $p_R$ also form an increasing sequence $ y_4 < \cdots < y_N$ in $(1,\infty)$ and $p_L$ is isometric to $p_R$ if and only if $(x_4, \ldots, x_N) = (y_4, \ldots, y_N)$. It follows that $\overline{\AdSPoly}_N$ is homeomorphic to $\RR^{N-3} \times \RR^{N-3}$ and $\AdSPoly_N$ is homeomorphic to $\RR^{N-3} \times \RR^{N-3}$ minus the diagonal. Therefore $\AdSPoly_N$ is homotopy equivalent to the sphere of dimension $N-4$.
\end{proof}

\begin{prop}
If $N \geq 6$, then $\HPPoly_N$ is connected and simply connected.
\end{prop}

\begin{proof}
Recall from Section~\ref{subsec:polyhedra-HP} that the space $\HPPoly_N$ identifies with the space of pairs $(p, V)$ where $p$ is a marked ideal $N$-gon in the hyperbolic plane and $V$ is a non-trivial infinitesimal deformation of $p$ considered up to the action of $T \PSL_2 \RR$. Using this action we may, in a unique way, place $(p, V)$ in standard position so that the first three vertices of $p$ are $x_1 = \infty, x_2 = 0, x_3 = 1$ and so that $V(x_1) = 0$, $V(x_2) = 0$, and $V(x_3) = 0$. The remaining $N-3$ tangent vectors are not all zero and their basepoints form an increasing sequence in $(1, \infty)$. It follows that $\HPPoly_N$ is homeomorphic to $T\RR^{N-3}$ minus the zero section. Therefore $\HPPoly_N$ is homotopy equivalent to the sphere of dimension $N-4$.
\end{proof}

\noindent As a corollary of Theorem~\ref{thm:main-HP} and this proposition we have:

\begin{cor}
\label{cor:topology}
The space of angle assignments $\Angles$ is connected and simply connected whenever the number of vertices $N \geq 6$.
\end{cor}

\subsection{$\PsiAdS$ is a local homeo}

Lemma~\ref{lem:PsiAdS-rigidity} says that for each triangulation $\Gamma \in\Graph(N, \gamma)$, the map $\Psi_\Gamma: \AdSPoly \to \RR^E$ is a local immersion at any ideal polyhedron $P$ whose $1$--skeleton is contained in $\Gamma$. We now deduce the following result.

\begin{lemma} \label{lem:local-diffeo}
$\PsiAdS: \AdSPoly \to \Angles$ is a local homeomorphism.
\end{lemma}

\begin{proof}
Given any $\Gamma \in \Graph(\Sigma_{0,N}, \gamma)$, we must first show that the dimension of $\Angles_\Gamma$ (if non-empty) is $2N -6$. The dimension of the convex cone $\Angles_\Gamma$ is determined by the rank of the $N$ equations of condition (ii). Assume first that $N$ is odd. Then these equations may be used to eliminate the the $N$ weights on the equator. Indeed if $\mathcal E_i$ denotes the equation of (ii) determined by the vertex $v_i$ of $\Gamma$, then treating indices cyclically we find that $$\mathcal E_{j+1} - \mathcal E_{j+2} + \cdots - \mathcal E_{j-1} + \mathcal E_j$$ is an equation which depends on (the weight at) the edge $e_j$ with endpoints $v_j$ and $v_{j+1}$ but on no other edge of the equator. This shows that the equations $\mathcal E_1, \ldots, \mathcal E_N$ have rank $N$ and the dimension of $\Angles_\Gamma$ is therefore $3N -6  - N = 2N-6$. Next if $N$ is even, we may only eliminate $N-1$ of the weights on the equator because all equatorial weights cancel in the alternating sum:
$$\mathcal E_1 - \mathcal E_2 + \cdots +\mathcal E_{N-1} - \mathcal E_N.$$
However, note that this sum is not trivial since it depends non-trivially on (the weight at) any edge whose two endpoints are an even number of edges apart along the equator. Since $\Gamma$ is a triangulation, there must exist some such edge. So the equations defined by condition (ii) in the definition of  $\Angles_\Gamma$ have rank $N$ in this case as well.

Next, for each triangulation $\Gamma$, let $V_\Gamma \subset \RR^{E(\Gamma)}$ be the subspace satisfying the equations of condition (ii). Since $V_\Gamma$ has dimension $2N-6$, as shown above, each of the maps $\PsiAdS_\Gamma$ is a local diffeomorphism at any polyhedron $P$ whose $1$--skeleton is a subgraph of $\Gamma$. The map $\PsiAdS$, pieced together from the $\PsiAdS_\Gamma$over all $\Gamma$, is an open map by the definition of the topology of the complex $\Angles$. Further, since each $\PsiAdS_\Gamma$ is a local diffeomorphism in a neighborhood of any point in the \emph{closure} of the stratum of $\AdSPoly$ defined by $\Gamma$, we have that $\PsiAdS$ is a local bijection to $\Angles$. It follows that $\PsiAdS$ is a local homeomorphism.
\end{proof}

Lemma~\ref{lem:local-diffeo} and Lemma~\ref{lem:PsiAdS-proper} imply that $\PsiAdS$ is a covering. Since for $N \geq 6$, $\AdSPoly$ is connected and $\Angles$ is connected and simply connected, we conclude that Theorem~\ref{thm:main-AdS-angles} holds when $N \geq 6$.

\subsection{The cases $N=4,5$}
Although the topology of $\Angles$ is slightly more complicated when $N=4,5$, the proof of Theorem~\ref{thm:main-AdS-angles} is straightforward in these cases. In the case $N=4$, the space $\AdSPoly$ is the space of marked (non-degenerate) ideal tetrahedra and has two components and the map $\PsiAdS$ is easily seen to be a homeomorphism. Indeed, an ideal tetrahedra in $\AdS^3$ is determined by its shape parameter (see Section~\ref{ads_background}); its dihedral angles may be determined directly from the shape parameter. Conversely, the shape parameter is determined by any two angles along edges emanating from a common vertex. Therefore an ideal tetrahedron is determined entirely by the local geometry near any ideal vertex. 

In the case $N=5$, both $\AdSPoly$ and $\Angles$ are homotopy equivalent to the circle. To show that the map $\PsiAdS$ is a homeomorphism, rather than some non-trivial covering, consider an ideal polyhedron $P$. We may cut $P$ into two ideal tetrahedra $T, T'$ along some interior triangular face $\Delta$. The tetrahedron $T$ is determined by the angles along the three edges emanating from any vertex of $T$, in particular the vertex not belonging to $\Delta$. These three angles are dihedral angles of $P$ as well, so it follows that the geometry of $T$ is determined by the dihedral angles of $P$. Similarly, the geometry of $T'$ is determined by the dihedral angles of $P$. Since there is exactly one way to glue $T$ and $T'$ back together (with the correct combinatorics), the geometry of $P$ is determined by its dihedral angles, i.e. $\PsiAdS$ is injective, and is therefore a homeomorphism.

\subsection{Proof of Theorem \ref{thm:main}}\label{sec:proof_main}
Finally, we prove Theorem~\ref{thm:main}. The equivalence of (C) and (H) is immediate from Theorems~\ref{thm:main-HP} and~\ref{thm:main-AdS-angles}.
We now show the equivalence of (H) and (S) using Theorem \ref{thm:main-AdS-angles} and Rivin's Theorem, discussed in Section \ref{sec:rivin_param}. 
Let $\Gamma \in \Graph(\Sigma_{0,N}, \gamma)$, and as usual let $E = E(\Gamma)$ denote the edges of $\Gamma$. First suppose $P\in \AdSPoly_\Gamma$, and let $\theta = \PsiAdS(P)\in \Angles_{\Gamma}$. For any $t > 0$, the weights $t \theta$ are also in $\Angles_\Gamma$. We choose $t > 0$ so that:
\begin{itemize}
 \item[(A)] for all edges $e \in E$, $t\theta(e)\in (-\pi,\pi)\setminus\{0\}$. 
 \item[(B)] for all of the finitely many simple cycles $c$ in $\Gamma^*$, the sum of the values of $t \theta$ along $c$ is greater than $-\pi$.
\end{itemize}
Note that any simple cycle $c$, as in (B) above, crosses the equator $\gamma$ at least twice. If $c$ crosses the equator $\gamma$ exactly twice, then this sum will either be zero, if $c$ bounds a face of $\Gamma^*$ (condition (ii) in the definition of $\Angles_\Gamma$), or positive if not (condition (iii)).
Noting that $t \theta(e) \in (-\pi, 0)$ if $e \in \gamma$ and $t \theta(e) \in (0, \pi)$ if not, we let $\theta': E \to (0,\pi)$ be defined by 
$$\theta'(e) =
\left\{
	\begin{array}{ll}
		t\theta(e)  & \mbox{if $e$ is not an edge of $\gamma$,} \\
		\pi+t\theta(e) & \mbox{if $e$ is an edge of $\gamma$.}
	\end{array}
\right.$$ Then $\theta'$ satisfies the three conditions of Rivin's Theorem and is therefore realized as the dihedral angles of some ideal polyhedron $P'$ in $\HH^3$. In the projective model for $\HH^3$, $P'$ is a polyhedron inscribed in the sphere with $1$--skeleton $\Gamma$.

Conversely, suppose $P'$ is an ideal polyhedron in $\HH^3$ with $1$--skeleton $\Gamma$. Then the dihedral angles $\theta':E \to (0,\pi)$ of $P'$ satisfy the three conditions of Rivin's Theorem. We define $\theta: E \to \RR$ by 
$$\theta(e) =
\left\{
	\begin{array}{ll}
		\theta'(e)  & \mbox{if $e$ is not an edge of $\gamma$,} \\
    \theta'(e)-\pi & \mbox{if $e$ is an edge of $\gamma$.}
	\end{array}
\right.$$

Then $\theta$ is easily seen to satisfy the three conditions in the definition of $\Angles_\Gamma$ and so by Theorem~\ref{thm:main-AdS-angles}, $\theta = \PsiAdS(P)$ for some $P \in \AdSPoly$. In the projective model for $\AdS^3$, $P$ is a polyhedron inscribed in the hyperboloid with $1$--skeleton~$\Gamma$.
This completes the proof of Theorem~\ref{thm:main}.

\begin{Remark}
Let $\Gamma$ be a planar graph and suppose $\Gamma$ is realized as the $1$--skeleton of some ideal polyhedron inscribed in the sphere. Note that $\Gamma$ may contain many different Hamiltonian cycles. Applying the above to each Hamiltonian cycle $\gamma$ shows the following: The components of the space of realizations of $\Gamma$ as the $1$--skeleton of a polyhedron inscribed in the hyperboloid (or similarly, the cylinder) is in one-one correspondence with the Hamiltonian cycles in $\Gamma$. 
\end{Remark}

\begin{appendix}
  
\section{Ideal polyhedra with dihedral angles going to zero}
\label{sec:transitional-proof}

We outline an alternative proof of Proposition~\ref{prop:imageinA-AdS} using transitional geometry ideas.
The argument uses Lemmas~\ref{lem:PsiAdS-proper} and~\ref{lem:PsiAdS-rigidity} to produce deformation paths of polyhedra with dihedral angles going to zero in a prescribed manner. Here is the basic idea: starting from an ideal polyhedron $P \in \AdSPoly$ with dihedral angles $\theta$, we deform $P$ so that the dihedral angles are proportional to $\theta$ and decrease toward zero. An appropriate rescaled limit of these collapsing polyhedra yields an ideal polyhedron $P_\infty'$ in $\HP^3$ whose (infinitesimal) dihedral angles are precisely $\theta$; we then conclude, via Proposition~\ref{prop:imageinA-HP}, that $\theta$ was in $\Angles$ to begin with. 

The main ingredient is the following proposition. Recall the projective transformations $\mathfrak a_t$ of Section~\ref{subsec:polyhedra-HP}, which when applied to (the projective model of) $\AdS^3$ yield $\HP^3$ in the limit as $t \to 0$.

\begin{Proposition}\label{prop:flat-limit}
Let $\Gamma \in\Graph(\Sigma_{0,N}, \gamma)$ and consider weights $\theta \in \RR^{E(\Gamma)}$ that satisfy conditions (i), (ii), and the following weaker version of (iii):
\begin{itemize}
\item[(iii'):]  If $e_1^*, \ldots, e_n^*$ form a simple circuit that does not bound a face of $\Gamma^*$, and such that exactly two of the edges are dual to edges of the equator, then $\theta(e_1^*) + \cdots + \theta(e_n^*) \neq 0$.
\end{itemize}
  Let $P_k$ be a sequence in $\AdSPoly_{\Gamma}$ with dihedral angles $t_k \theta$ such that $t_k \to 0$. Then:
\begin{enumerate}
\item $P_k$ converges to an ideal $N$-gon $P_\infty$ in the hyperbolic plane.
\item $\mathfrak a_{t_k} P_k$ converges to an ideal polyhedron $P_\infty'$ in $\HP^3$ with $1$--skeleton $\Gamma$ and infinitesimal dihedral angles $\theta$.
\end{enumerate}
\end{Proposition}
\noindent We briefly mention an analogue of the proposition in setting of quasi-Fuchsian hyperbolic three-manifolds. The first conclusion of the proposition can be seen as an analogue of Series' theorem \cite{ser_lim}, which states that when the bending data of a sequence of quasi-Fuchsian representations goes to zero in a controlled manner, the convex cores collapse to a Fuchsian surface. The second part is the analogue of work of Danciger--Kerckhoff \cite{dan_tra} showing that after application of appropriate projective transformations (in our notation, the $\mathfrak a_t$), the collapsing convex cores of such quasi-Fuchsian representations converge to a convex core in half-pipe geometry. 

\begin{proof}
We adapt the proof of Lemma~\ref{lem:PsiAdS-proper} (properness of the map $\PsiAdS$). As in that proof, we may again assume that the ideal vertices $(v^L_{1,k}, v^R_{1,k}), \ldots, (v^L_{N,k}, v^R_{N,k})$ of $P_k$ satisfy that:
\begin{itemize}
\item $v^L_{1,k} = v^R_{1,k} = 0$, $v^L_{2,k} = v^R_{2,k} = 1$, $v^L_{3,k} = v^R_{3,k} = \infty$;
\item For each $i \in \{1,\ldots N\}$, $v^L_{i,k} \to v^L_{i,\infty}$ and $v^R_{i,k} \to v^R_{i,\infty}$; and
\item $v^L_{i, \infty} = v^L_{i+1, \infty}$ if and only if $v^R_{i, \infty} = v^R_{i+1, \infty}$.
\end{itemize}

Therefore, we again find that the limit $P_\infty$ of $P_k$ (in this normalization) is a convex ideal polyhedron in $\AdS^3$, possibly of fewer vertices, and possibly degenerate (i.e. lying in a two-plane). The dihedral angle at an edge $e$ of $P_\infty$ is again the sum of $\theta_\infty(e')$ over all edges $e'$ of $\Gamma$ which collapse to~$e$, where in this case $\theta_\infty = 0$. Therefore all dihedral angles of $P_\infty$ are zero and we have that $P_\infty$ is an ideal polygon lying in the hyperbolic plane $\plane$ containing the ideal triangle $\Delta_0$ spanned by $(0,0), (1,1)$, and $(\infty, \infty)$. 
To prevent collapse, we apply the projective transformations $\mathfrak a_{t_k}$ to the $P_k$. 
\begin{claim}
Up to taking a subsequence (in fact not necessary), the vertices $\mathfrak a_{t_k} v_{i,k}$ converge to points $v_{i,\infty}'$ in the ideal boundary $\partial_\infty \HP^3$. 
\end{claim}
\begin{proof}
 This can be seen from the following simple compactness statement, which may be verified by induction: Given $M \geq 1$ and $\Theta > 0$, there exists two smooth families of space-like planes $\mathcal Q_+(t)$ and $\mathcal Q_-(t)$, defined for $t \geq 0$, such that 
 \begin{itemize}
 \item $\mathcal Q_+(0) = \mathcal  Q_-(0) = \plane$.
\item $\mathcal Q_+(t)$ and $\mathcal Q_-(t)$ are disjoint for $t > 0$ and their common perpendicular is a fixed time-like line $\alpha$ (independent of $t$). 
\item The time-like distance (along $\alpha$) between $\mathcal Q_+(t)$ and $\mathcal Q_-(t)$ is $\operatorname{O}(t)$.
\item Any space-like convex connected ideal polygonal surface in $\AdS^3$ for which $\Delta_0$ is (contained in) a face, which has at most $M$ faces, and all of whose dihedral angles are bounded by $t \Theta$ lies to the past of $\mathcal Q_+(t)$ and to the future of $\mathcal Q_-(t)$.
 \end{itemize}
 The first three conditions above imply that the limit of $\mathfrak a_t \mathcal Q_+(t)$ and $\mathfrak a_t \mathcal Q_-(t)$ as $t \to 0$ are two disjoint non-degenerate planes $\mathcal Q_+'$ and $\mathcal Q_-'$ in $\HP^3$. Therefore, the limit of $\mathfrak a_{t_k} P_k$ must, after extracting a subsequence if necessary, converge to some polyhedron in $\HP^3 \cup \partial_\infty \HP^3$ lying below $\mathcal Q_+'$ and above $\mathcal Q_-'$.
 \end{proof}

As in the proof of Lemma~\ref{lem:PsiAdS-proper}, the limit of $\mathfrak a_{t_k} P_k$ is the convex hull $P_\infty'$ of $v_{1,\infty}', \ldots, v_{n, \infty}'$ in $\HP^3$. The $1$--skeleton $\Gamma'$ of $P_\infty'$ is obtained from the original $1$--skeleton $\Gamma$ by collapsing some edges to vertices and some faces to edges or vertices. A simple argument in HP geometry gives that:
\begin{lemma}
Given $e' \in \Gamma'$, the infinitesimal dihedral angle $\theta_\infty'(e')$ of $P_\infty'$ at $e'$ is the sum of $\theta(e) = \frac{d}{dt} t \theta(e)\big|_{t=0}$ over all edges $e$ which collapse to~$e'$.
\end{lemma}

Next, consider the projection $\varpi: \operatorname{HP}^3 \to \plane$. Note that $\varpi(v_{i,\infty}') = v_{i,\infty}$. Let $\mathcal H$ denote the HP horo-cylinder which is the inverse image under $\varpi$ of a small horocycle in $\plane$ centered at a vertex $v_{i, \infty}$ of $P_\infty$. The metric on $\mathcal H$ inherited from $\HP^3$ is flat and degenerate; it is the pull-back under $\varpi$ of the metric on a horocycle. The intersection of $\mathcal H$ with $P_\infty'$ is a convex polygon $q$ in $\mathcal H$. The infinitesimal angles at vertices of $q$ are the same as the infinitesimal dihedral angles of the corresponding edges of $P_\infty'$. Note that the vertices of $q$ are the intersection with $\mathcal H$ of all edges emanating from the ideal points $v_{j, \infty}'$ such that $v_{j, \infty} = v_{i,\infty}$. A simple calculation in this degenerate plane shows that:

\begin{lemma}
The infinitesimal angles of $q$ sum to zero.
\end{lemma}

Now, suppose, for contradiction, that $v_{i+1, \infty} = v_{i, \infty}$. Then, the vertices of $q$ correspond to a path $c'$ of edges of $\Gamma'$ whose inverse image under the collapse is a path $c$ of edges in $\Gamma$ which do not bound a face of $\Gamma^*$. It follows from the above that the sum of $\theta(e)$ over the edges $e$ of the path $c$ is zero, contradicting the condition~(iii').
\end{proof}

\begin{remark}
This argument also works in the context of hyperbolic ideal polyhedra with dihedral angles going to zero and $\pi$ at a controlled rate.
\end{remark}
\begin{remark}
Assuming the stronger condition~(iii) on $\theta$, the limiting ideal polygon $P_\infty$ must be the unique minimum of the length function $\ell_\theta$ over the space $\poly_N$ of marked ideal polygons. See the proof of Theorem~\ref{thm:main-HP}.
\end{remark}

\begin{proof}[Outline of alternative proof of Proposition~\ref{prop:imageinA-AdS}]
Let $\Gamma \in \Graph(\Sigma_{0,N}, \gamma)$ and suppose $P \in \AdSPoly_\Gamma$ such that the dihedral angles $\theta = \PsiAdS(P) \in \RR^{E(\Gamma)}$ violate condition (iii) in the definition of $\Angles_\Gamma$. We argue by contradiction. First we show that there are nearby weights $\theta'$ satisfying conditions (i), (ii), as well as condition (iii') of Proposition~\ref{prop:flat-limit} above and so that at least one of the angle sum expressions of (iii') is strictly \emph{negative}. This may already be the case for $\theta$. If not, then there is at least one angle sum expressions as in (iii) which evaluates to zero, and we will perturb. In the case that $\Gamma$ is a triangulation, it is simple to verify that that none of the angle sum expressions in condition (iii) is locally constant when the equations of condition (ii) are satisfied, and therefore a nearby $\theta'$ exists as desired, since (iii') consists of only finitely many conditions . If $\Gamma$ is not a triangulation, then it could be the case that an angle sum expression as in condition (iii) is constant equal to zero on the entire space of weights satisfying (ii) (see Figure~\ref{fig:strange-example} for an example!). A quick study of the possible ways such degenerate behavior can happen reveals that it is always possible to add a small number of edges (at most one for each angle sum expression of (iii) which evaluates to zero for $\theta$) with very small positive weights, while perturbing the other weights slightly, to produce $\theta'$ as desired. 

Next, by Lemma~\ref{lem:local-diffeo} (which was a simple consequence of Lemma~\ref{lem:PsiAdS-rigidity}, independent of Proposition~\ref{prop:imageinA-AdS}), there is an ideal polyhedron $P' \in \AdSPoly$, close to $P$, so that $\PsiAdS(P') = \theta'$. Now, consider the path of weights $t\theta'$, defined for $t > 0$. Lemma~\ref{lem:local-diffeo} implies that there exists a path $P_t$ in $\AdSPoly$ such that $\PsiAdS(P_t) = t \theta'$, defined at least for $t$ close to one. In fact, the path $P_t$ may be defined for all $1 \geq t > 0$. Indeed if the limit as $t \to T > 0$ of $P_t$ failed to exist, then the proof of Lemma~\ref{lem:PsiAdS-proper} would imply that $\PsiAdS(P_t)$ either goes to infinity or limits to an element of $\RR^{E(\Gamma)}$ for which some angle sum expression as in (iii) is exactly zero, impossible since the limit as $t \to T$ of $\PsiAdS(P_t)$ is, of course, equal to~$T \theta$.
Hence, we may apply Proposition~\ref{prop:flat-limit} to the path $P_t$. The result is an ideal polyhedron $P_\infty' \in \HPPoly$ whose infinitesimal dihedral angles are precisely $\theta'$. This contradicts Proposition~\ref{prop:imageinA-HP} since $\theta'$ does not satisfy (iii).

 \end{proof}
 
 \begin{figure}
 \def\svgwidth{2.5in}
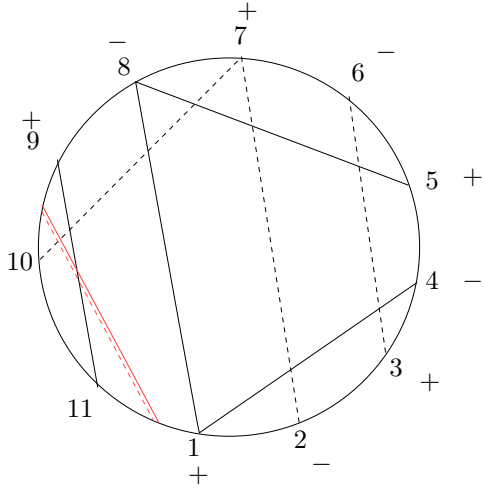
\caption{The black edges (including circular arcs, and dotted edges) form a three-connected graph which contains a Hamiltonian path (the circular arcs), but which is never realized as the $1$--skeleton of a convex ideal polyhedron in $\AdS^3$. The red path determines a path in the dual graph as in condition (iii) for which the angle sum is identically zero over any systems of weights satisfying (ii). Indeed the angle sum is precisely the alternating sum of the terms in the vertex equations (for vertices 1--9) with signs as labeled in the diagram. \label{fig:strange-example}}
 \end{figure}
  
\end{appendix}

\bibliographystyle{amsplain}
\bibliography{sample2,adsbib,sample1}

\end{document}